\documentclass[11pt,a4paper,reqno]{amsart}

\usepackage{amsmath,amsfonts,amsthm,amsopn,amssymb,cite,mathrsfs}
\usepackage{cancel,marginnote}

\usepackage{color,enumitem,graphicx}
\usepackage[colorlinks=true,urlcolor=blue,
citecolor=red,linkcolor=blue,linktocpage,pdfpagelabels]{hyperref}
\usepackage[english]{babel}

\usepackage[left=2.8cm,right=2.8cm,top=2.8cm,bottom=2.8cm]{geometry}

\usepackage{showkeys}
\usepackage[hyperpageref]{backref}

\makeindex

\numberwithin{equation}{section}
\newtheorem{theorem}{Theorem}[section]
\newtheorem{lemma}{Lemma}[section]

\newtheorem{remark}{Remark}[section]

\newtheorem{definition}{Definition}[section]

\title[two critical exponents]
{Ground State Solutions for local-nonlocal Shr\"{o}dinger equations in the presence of two critical exponents}

\thanks{Yu Su is supported by the Natural Science Research Project of Anhui Educational Committee (Grant No. 2023AH040155).}


\begin{document}
\maketitle

\centerline{\scshape Yu Su$^{1}$, Hichem Hajaiej$^{2,*}$, Hongxia Shi$^{3}$,~~}
\medskip
{\footnotesize

\centerline{1. School of Mathematics and Big Data, Anhui University of Science and Technology,}
\centerline{Huainan, Anhui 232001, China}
\centerline{yusumath@qq.com}

\centerline{2. Department of Mathematics, California State University,}
\centerline{Los Angeles, California, USA}
\centerline{hhajaie@calstatela.edu}

\centerline{3. 
School of Mathematics and Computational Science, Hunan First Normal University,}
\centerline{Changsha, Hunan 410205, China}
\centerline{shihongxia5617@163.com}

}

\begin{abstract}
In this paper,
we address the existence of ground state solutions for Schr\"{o}dinger equations in the presence of local and nonlocal operators and two critical nonlinearities associated with each operator.
The situation is completely solved in the critical case  both in the local and in the nonlocal settings. 
However,
methods developed in these cases cannot extend to local-nonlocal operators when we have two critical power nonlinearities.
This unprecedented situation in PDEs presents some challenges, and its resolution will open the door to solve similar problems.
Our approach is essentially based on a subtle generalization of the Lieb translation theorem.
We will also discuss the positivity of the ground states as well as their regularity.
\end{abstract}

\section{Introduction}
We study the following  Schr\"{o}dinger equation with the local-nonlocal operator
\begin{equation}\label{S}
\begin{aligned}
-\lambda\Delta u
+\mu(-\Delta)^{s}u=f(u), \ \ x\in \mathbb{R}^{N},
\end{aligned}
\tag{$S_{\lambda,\mu}$}
\end{equation}
where $N\geqslant3$ and  $0<s<1$.
And $(-\Delta)^{s}$ is the so-called fractional
Laplacian, which can be defined, for any $u:\mathbb{R}^{N}\to\mathbb{R}$ smooth enough, by setting
$
\mathcal{F}((-\Delta)^{s}u)(\xi)
=
|\xi|^{2s}
\mathcal{F}(u)(\xi)
$,
$\xi\in\mathbb{R}^{N}$,
where $\mathcal{F}$ represents the Fourier transform, see \cite{Nezza-Palatucci-Valdinoci2012BSM}.

For $\lambda,\mu>0$,
equation \eqref{S} arises in the population dynamics model with both classical and nonlocal diffusion is discussed by Dipierro-Lippi-Valdinoci \cite{Dipierro-Lippi-Valdinoci2022AIHP}, which is motivated by the biologically relevant situation of a population following long-jump foraging patterns alternated with focused searching strategies at small scales.
Moreover, Biagi-Dipierro-Valdinoci-Vecchi \cite{Biagi-Dipierro-Valdinoci-Vecchi2021CPDE} pointed out that equation \eqref{S} can apply to study the different types of ``regional" or ``global" restrictions that reduce the spreading of a pandemic disease.
Furthermore,
Dipierro-Valdinoci \cite{Dipierro-Valdinoci2021PA} introduced \eqref{S} in the description of an ecological niche for a mixed local and nonlocal dispersal.
For more applications,
we refer the reader to Bernstein-type regularity results
\cite{Cabre-Dipierro-Valdinoci2022ARMA},
the Aubry-Mather theory for sums of different fractional Laplacians
\cite{de la Llave-Valdinoci2009Poincare}, numerics\cite{Biswas-Jakobsen-Karlsen2010SIAM-J-N-A},
probability and stochastics
\cite{Mimica2016PLMS,ChenZQ-Kim-Song-Vondracek2012TAMS}.

For $\lambda=1$ and $\mu=0$,
equation \eqref{S} becomes the classical nonlinear Schr\"{o}dinger equation 
\begin{equation}\label{S1}
\begin{aligned}
-\Delta u=f(u), \ \ x\in \mathbb{R}^{N},
\end{aligned}
\tag{$S$}
\end{equation}
which arises in quantum mechanics and quantum field theory.
We recall that the Sobolev critical exponent is defined by
\begin{equation*}
\begin{aligned}
2^{*}:=\frac{2N}{N-2}.
\end{aligned}
\end{equation*}
For subcritical nonlinearities,
there are two classical famous conditions to ensure the existence of solutions:
the Ambrosetti-Rabinowitz type conditions (in  short for A-R conditions) and Berestycki-Lions type conditions (in short for B-L conditions).
In 1973,
Ambrosetti-Rabinowitz \cite{Ambrosetti-Rabinowitz1973JFA} showed the existence of the solutions \textcolor{red}{to} equation \eqref{S1} with A-R conditions.
In 1983,
for the subcritical case,
Berestycki-Lions \cite{Berestycki-Lions1983ARMA} showed that B-L conditions are almost necessary and sufficient condition for the existence of ground state solution of equation \eqref{S1}.
Moreover,
the uniqueness result was studied in \cite{Serrin-M.X.Tang2000IUMJ}.
For the critical case $f(u)=|u|^{2^{*}-2}u$,
Aubin \cite{Aubin1976JDG} and Talenti \cite{Talenti1976AMPA} considered the existence result and the exact formula of radially symmetric positive ground state solution.
For more related works,
we refer the reader to  \cite{lions1984Pioncare}.

For $\lambda=0$ and $\mu=1$,
equation \eqref{S} becomes the fractional Schr\"{o}dinger equation 
\begin{equation}\label{FS}
\begin{aligned}
(-\Delta)^{s} u=f(u), \ \ x\in \mathbb{R}^{N}.
\end{aligned}
\tag{$FS$}
\end{equation}
When $s\in(\frac{1}{2},1)$,
it is a fractional Schr\"{o}dinger equation which describes the energy and the momentum of non-relativistic fractional quantum-mechanical particle \cite{Laskin2000PLA,Laskin2002PRE}. 
The fractional Sobolev critical exponent is defined as follows
\begin{equation*}
\begin{aligned}
2_{s}^{*}:=\frac{2N}{N-2s}.
\end{aligned}
\end{equation*}
For the subcritical cases (A-R conditions and B-L conditions), the existence of solution for equation \eqref{FS} follows from variational arguments.
Caffarelli-Silvestre \cite{Caffarelli-Silvestre2007PDE} developed a powerful extension method that transfer the nonlocal problem \eqref{FS} into a local one on a half-space.
Motivated by numerous and various applications, a lot of significant advances have been made.
See, e.g.,
\cite{Hajaiej2013JMAA,Hajaiej2013Poincare,Hajaiej2014AA} for the existence
and symmetry of minimizers,
and
\cite{Cho-Hajaiej-Hwang-Ozawa2013FE,Cho-Hajaiej-Hwang-Ozawa2014CPAA} for the Cauchy problem.
For the critical case,
Lieb \cite{Lieb1983AM} and Cotsiolis-Tavoularis \cite{Cotsiolis-Tavoularis2004JMAA} considered the existence result and the exact formula of radially symmetric positive ground state solution. 
For more related works, we refer the reader to \cite{Frank-Lenzmann2013Acta,Dipierro-Montoro-Peral-Sciunzi2016CVPDE}.

Luo-Hajaiej \cite{LuoTJ-Hajaiej2022ANS} considered the following  nonlinear Schr\"odinger equations with mixed fractional Laplacians ($0<s_{1}<s_{2}\leqslant1$)
\begin{equation}\label{MFS}
\begin{aligned}
(-\Delta)^{s_{1}} u
+(-\Delta)^{s_{2}}u+\lambda u=|u|^{p-2}u, \ \ x\in \mathbb{R}^{N},
\end{aligned}
\tag{$MFS$}
\end{equation}
When $s_{2}=1$ and $\lambda=0$,
then equation \eqref{MFS} reduces to equation \eqref{S}.
For $\lambda$ unknown,
there are two mass critical exponents
\begin{equation*}
\begin{aligned}
2+\frac{s_{1}}{N}~~\mathrm{and}~~2+\frac{s_{2}}{N}.
\end{aligned}
\end{equation*}
Luo-Hajaiej \cite{LuoTJ-Hajaiej2022ANS} investigated the existence of ground state normalized solutions for $p\in (2,2+\frac{s_{1}}{N})$,
$p=2+\frac{s_{1}}{N}$
and
$p\in (2+\frac{s_{1}}{N},2+\frac{s_{2}}{N})$ with some suitable $L^{2}$ norm,
and the nonexistence of ground state normalized solutions for $p=2+\frac{s_{2}}{N}$.
For $p\in(2+\frac{s_{2}}{N},\frac{2N}{N-2s_{2}})$,
Chergui-Gou-Hajaiej \cite{Chergui-Gou-Hajaiej2023CVPDE} considered the existence and dynamics  results for equation \eqref{MFS}.
Su-Valdinoci-Wei-Zhang
\cite{Su-Valdinoci-Wei-Zhang2022MZ,Su-Valdinoci-Wei-Zhang2025JDE}
and Dipierro-Su-Valdinoci-Zhang
\cite{Dipierro-Su-Valdinoci-Zhang2025DCDS} show the regularity results of equation \eqref{MFS} with $s_{2}=1$.
Biagi-Dipierro-Valdinoci-Vecchi
\cite{Biagi-Dipierro-Valdinoci-Vecchi2025NoDEA}
established a Brezis-Nirenberg type result for $s_{2}=1$. 
For more related works, we also refer to \cite{ChenHY-Bhakta-Hajaiej2022JDE,Hajaiej-Perera2022DIE,Bahrouni-Guo-Hajaiej2025AA,Giovannardi-Mugnai-Vecchi2023JMAA,Cangiotti-Caponi-Maione-Vitillaro2024FFAC,Cangiotti-Caponi-Maione-Vitillaro2023Milan}.

We recall the case $\lambda=\mu=1$ and  $f(u)=|u|^{p-2}u$. Then equation \eqref{S} is
\begin{equation}\label{P}
\begin{aligned}
-\Delta u
+(-\Delta)^{s}u =|u|^{p-2}u, \ \ x\in\mathbb{R}^{N}.
\end{aligned}
\tag{$P$}
\end{equation}

\begin{theorem}{\rm 
\cite{BaoQQ2024Malaysian,Chergui-Gou-Hajaiej2023CVPDE}}\label{Theorem1.1}
Let $N\geqslant3$ and $0<s<1$.
Then we have the following results.

\begin{enumerate}
\item [(i)] If  $p=2_{s}^{*}$ or $p=2^{*}$,
then equation \eqref{P} has no non-trivial solution.

\item [(ii)] If $p\in (2_{s}^{*},2^{*})$, then equation \eqref{P} has a radial ground state solution.
\end{enumerate}

\end{theorem}

\begin{remark}
From Theorems \ref{Theorem1.1}, it follows that $2_{s}^{*}$ and $2^{*}$ are the lower and upper critical exponents,
respectively.
This is a new phenomenon for PDEs that requires new tools we will develop in this paper.
\end{remark}

More precisely,
we will investigate the existence of solutions to equation \eqref{S} in the presence of the lower and upper critical exponents.

Namely,
we will focus on the nonlinearity:
\begin{equation*}
\begin{aligned}
f(u)=|u|^{2_{s}^{*}-2}u+\beta|u|^{p-2}u+|u|^{2^{*}-2}u
\end{aligned}
\end{equation*}
with $p\in(2_{s}^{*},2^{*})$, then equation \eqref{S} is
\begin{equation}\label{D}
\begin{aligned}
-\Delta u
+(-\Delta)^{s}u =|u|^{2_{s}^{*}-2}u+\beta|u|^{p-2}u+|u|^{2^{*}-2}u, \ \ x\in\mathbb{R}^{N}.
\end{aligned}
\tag{$D$}
\end{equation}
\begin{theorem}\label{Theorem1.2}
Let $N\geqslant3$, $0<s<1$ and $p\in(2_{s}^{*},2^{*})$. Then there exists $\bar{\beta}\in(0,+\infty)$ such that for any $\beta>\bar{\beta}$, equation \eqref{D} has a radial ground state solution.
\end{theorem}

Moreover, we establish an $L^{\infty}$ estimate result of non-negative solutions for
\begin{equation}\label{F}
\begin{aligned}
-\Delta u
+(-\Delta)^{s}u=f(u), \ \ x\in\mathbb{R}^{N},
\end{aligned}
\tag{$F$}
\end{equation}
where the nonlinearity $f(\cdot)$ satisfies the following condition
\begin{enumerate}
\item [$(F_{1})$] There exists $C>0$ such that for every $t\in \mathbb{R}$ such that
\begin{equation*}
\begin{aligned}
|f(t)t|\leqslant C(t^{2_{s}^{*}}+t^{2^{*}}).
\end{aligned}
\end{equation*}
\end{enumerate}

\begin{theorem}\label{Theorem1.3}
Let $N\geqslant3$, $0<s<1$ and $(F_{1})$ hold. Then the non-negative solutions of equations \eqref{F} are in $L^{\infty}(\mathbb{R}^{N})$.
\end{theorem}

This paper is organized as follows.
In Section 2,
we present the functional setting.
In Section 3,
we prove the Lieb translation theorem.
In Section 4, we show the Mountain pass geometry and Nehari manifold.
In Section 5,
we show Theorem \ref{Theorem1.2},
i.e., the existence of ground state solution solution.
In Section 6, we prove Theorem \ref{Theorem1.3}, i.e., the regularity of the solutions.

\section{Sobolev Space}
Let us define the homogeneous Sobolev space as follows 
\begin{equation*}
\begin{aligned}
D^{1,2}(\mathbb{R}^{N})
=
\{u\in L^{2^{*}}(\mathbb{R}^{N})|
|\nabla u|\in L^{2}(\mathbb{R}^{N})\},
\end{aligned}
\end{equation*}
its semi-norm is taken as
\begin{equation*}
\begin{aligned}
\|u\|_{D^{1,2}
(\mathbb{R}^{N})}^{2}
=\int_{\mathbb{R}^{N}}
|\nabla u|^{2}
\mathrm{d}x.
\end{aligned}
\end{equation*}
For $N\geqslant 3$ and $s\in(0,1)$, let $D^{s,2}(\mathbb{R}^{N})$ be the homogeneous fractional Sobolev space, which is the completion of $C_{0}^{\infty}(\mathbb{R}^{N})$ with the semi-norm
\begin{equation*}
\begin{aligned}
\|u\|_{D^{s,2}(\mathbb{R}^{N})}^{2}:=
\int_{\mathbb{R}^{N}}\int_{\mathbb{R}^{N}}\frac{|u(x)-u(y)|^{2}}{|x-y|^{N+2s}}\mathrm{d}x\mathrm{d}y.
\end{aligned}
\end{equation*}
The mixed Sobolev space $X$ defined by the completion of $C^{\infty}_{0}(\mathbb{R}^{N})$ under the semi-norm
\begin{equation*}
\begin{aligned}
\|u\|_{X}^{2}:=
\int_{\mathbb{R}^{N}}
|\nabla u|^{2}
\mathrm{d}x
+
\int_{\mathbb{R}^{N}}
\int_{\mathbb{R}^{N}}
\frac{|u(x)-u(y)|^{2}}{|x-y|^{N+2s}}
\mathrm{d}x
\mathrm{d}y.
\end{aligned}
\end{equation*}

\begin{lemma}\label{Lemma2.1}
$X\hookrightarrow D^{1,2}(\mathbb{R}^{N})$ and 
$X\hookrightarrow D^{s,2}(\mathbb{R}^{N})$.
\end{lemma}
\begin{proof}
It is easy to see that
\begin{equation*}
\begin{aligned}
\|u\|_{D^{1,2}(\mathbb{R}^{N})}^{2}
\leqslant \|u\|_{X}^{2},
\end{aligned}
\end{equation*}
and
\begin{equation*}
\begin{aligned}
\|u\|_{D^{s,2}(\mathbb{R}^{N})}^{2}
\leqslant \|u\|_{X}^{2}.
\end{aligned}
\end{equation*}
These show $X\hookrightarrow D^{1,2}(\mathbb{R}^{N})$ and 
$X\hookrightarrow D^{s,2}(\mathbb{R}^{N})$.
\end{proof}

\begin{lemma}\label{Lemma2.2}
$X$ is continuously embedding $L^{t}(\mathbb{R}^{N})$, where $t\in [2_{s}^{*},2^{*}]$, and $2_{s}^{*}:=\frac{2N}{N-2s}$ and $2^{*}:=\frac{2N}{N-2}$.
\end{lemma}

\begin{proof}
Using H\"{o}lder's inequality, we have
\begin{equation*}
\begin{aligned}
\int_{\mathbb{R}^{N}}
|u|^{t}
\mathrm{d}x
\leqslant
\left(
\int_{\mathbb{R}^{N}}
|u|^{2_{s}^{*}}
\mathrm{d}x
\right)
^{\frac{(tN-2N-2t)(N-2s)}{4N(s-1)}}
\left(
\int_{\mathbb{R}^{N}}
|u|^{2^{*}}
\mathrm{d}x
\right)
^{\frac{(tN-2N-2ts)(N-2)}{4N(1-s)}}.
\end{aligned}
\end{equation*}
From Lemma \ref{Lemma2.1}, we know\begin{equation*}
\begin{aligned}
\left(
\int_{\mathbb{R}^{N}}
|u|^{2_{s}^{*}}
\mathrm{d}x
\right)^{\frac{2}{2_{s}^{*}}}
\leqslant
\|u\|_{D^{s,2}(\mathbb{R}^{N})}^{2}
\leqslant \|u\|_{X}^{2},
\end{aligned}
\end{equation*}
and
\begin{equation*}
\begin{aligned}
\left(
\int_{\mathbb{R}^{N}}
|u|^{2^{*}}
\mathrm{d}x
\right)^{\frac{2}{2^{*}}}
\leqslant
\|u\|_{D^{1,2}(\mathbb{R}^{N})}^{2}
\leqslant \|u\|_{X}^{2}.
\end{aligned}
\end{equation*}
Then we get
\begin{equation*}
\begin{aligned}
\int_{\mathbb{R}^{N}}
|u|^{t}
\mathrm{d}x
\leqslant&
\left(
\int_{\mathbb{R}^{N}}
|u|^{2_{s}^{*}}
\mathrm{d}x
\right)
^{\frac{(tN-2N-2t)(N-2s)}{4N(s-1)}}
\left(
\int_{\mathbb{R}^{N}}
|u|^{2^{*}}
\mathrm{d}x
\right)
^{\frac{(tN-2N-2ts)(N-2)}{4N(1-s)}}\\
\leqslant&\|u\|_{X}^{t}<\infty.
\end{aligned}
\end{equation*}
The proof is completed.
\end{proof}

\begin{lemma}\label{Lemma2.3}
{\rm \cite{Nezza-Palatucci-Valdinoci2012BSM}}
Let
$s\in(0,1]$
and $N>2s$. 
Then there exists a constant 
$S_{s}>0$ 
such that for
any 
$u\in D^{s,2}(\mathbb{R}^{N})$,
\begin{equation*}
\begin{aligned}
\|u\|_{L^{2_{s}^{*}}(\mathbb{R}^{N})}^{2}
\leqslant
S_{s}^{-1}
\|u\|^{2}_{D^{s,2}(\mathbb{R}^{N})}.
\end{aligned}
\end{equation*}
\end{lemma}

\section{Lieb's translation theorem}
In this section, 
we recall the Lieb translation theorem and we prove an extention of it.

\begin{lemma}\label{Lemma3.1}
Let $N\geqslant 3$, $s\in(0,1)$ and $q\in(2_{s}^{*},2^{*})$. Then the following inequality holds
\begin{equation*}
\begin{aligned}
\int_{\mathbb{R}^{N}}|u|^{q}\mathrm{d}x\leqslant2C(N+1)^{2}
\left(\sup_{z\in \mathbb{R}^{N}}
\int_{B(z,1)}
|u|^{q}
\mathrm{d}x\right)^{\frac{q-2}{q}}
\|u\|_{X}^{2},
\end{aligned}
\end{equation*}
for all $u\in X$.
\end{lemma}

\begin{proof}
Let $u\in X$ and $q\in(2_{s}^{*},2^{*})$.
From Lemma \ref{Lemma2.3},
we have that
\begin{equation}\label{3.1}
\begin{aligned}
&\int_{B(z,1)}|u|^{q}\mathrm{d}x\\
\leqslant&
\left(
\int_{B(z,1)}
|u|^{2_{s}^{*}}
\mathrm{d}x
\right)^{\frac{2^{*}-q}{2^{*}-2_{s}^{*}}}
\left(
\int_{B(z,1)}
|u|^{2^{*}}
\mathrm{d}x
\right)^{\frac{q-2_{s}^{*}}{2^{*}-2_{s}^{*}}}\\
\leqslant&
\left(
S_{s}^{-1}
\int_{B(z,1)}\int_{B(z,1)}\frac{|u(x)-u(y)|^{2}}{|x-y|^{N+2s}}\mathrm{d}y\mathrm{d}x
\right)^{\frac{2_{s}^{*}}{2}\frac{2^{*}-q}{2^{*}-2_{s}^{*}}}
\left(
S_{1}^{-1}
\int_{B(z,1)}
|\nabla u|^{2}
\mathrm{d}x
\right)^{\frac{2^{*}}{2}\frac{q-2_{s}^{*}}{2^{*}-2_{s}^{*}}}\\
\leqslant&
S_{s}^{-\frac{2_{s}^{*}}{2}\frac{2^{*}-q}{2^{*}-2_{s}^{*}}}
S_{1}^{-\frac{2^{*}}{2}\frac{q-2_{s}^{*}}{2^{*}-2_{s}^{*}}}
\left(
\int_{B(z,1)}\int_{B(z,1)}\frac{|u(x)-u(y)|^{2}}{|x-y|^{N+2s}}\mathrm{d}y\mathrm{d}x
+
\int_{B(z,1)}
|\nabla u|^{2}
\mathrm{d}x\right)^{\frac{q}{2}}.
\end{aligned}
\end{equation}
Applying \eqref{3.1},
we know that
\begin{equation*}
\begin{aligned}
&\int_{B(z,1)}|u|^{q}\mathrm{d}x\\
=&
\left(\int_{B(z,1)}|u|^{q}\mathrm{d}x\right)^{\frac{2}{q}}
\left(\int_{B(z,1)}|u|^{q}\mathrm{d}x\right)
^{\frac{q-2}{q}}\\
\leqslant&
C
\left(
\int_{B(z,1)}\int_{B(z,1)}\frac{|u(x)-u(y)|^{2}}{|x-y|^{N+2s}}\mathrm{d}y\mathrm{d}x
+
\int_{B(z,1)}
|\nabla u|^{2}
\mathrm{d}x\right)
\left(\int_{B(z,1)}|u|^{q}\mathrm{d}x\right)
^{\frac{q-2}{q}}.
\end{aligned}
\end{equation*}
Covering $\mathbb{R}^{N}$ by balls of radius $1$,
in such a way that each point of $\mathbb{R}^{N}$ is contained in at most $N+1$ balls.
We find that
\begin{equation*}
\begin{aligned}
\int_{\mathbb{R}^{N}}|u|^{q}\mathrm{d}x
\leqslant C(N+1)
\left(\sup_{z\in \mathbb{R}^{N}}
\int_{B(z,1)}
|u|^{q}
\mathrm{d}x\right)^{\frac{q-2}{q}}
\|u\|_{X}^{2}.
\end{aligned}
\end{equation*}
\end{proof}

\begin{theorem}\label{Theorem3.1}
{\rm (Lieb's translation theorem)}
Let $N\geqslant3$, $0<s<1$, $q\in (2_{s}^{*},2^{*})$, and $\{u_{n}\}$ be a bounded sequence in $X$ satisfying:
\begin{equation*}
\begin{aligned}
\lim\limits_{n\to\infty}\int_{\mathbb{R}^{N}}|u_{n}|^{q}\mathrm{d}x>0.
\end{aligned}
\end{equation*}
Then there exists $\{z_{n}\}\subset \mathbb{R}^{N}$ such that $\{\bar{u}_{n}:=u_{n}(x+z_{n})\}$ convergence strongly to $\bar{u}\not\equiv0$ in $L^{q}_{loc}(\mathbb{R}^{N})$.
\end{theorem}

\begin{proof}
Note that $\{u_{n}\}$ is a bounded sequence in $X$.
Up to a subsequence,
we assume that
\begin{equation*}
\begin{aligned}
u_{n}\rightharpoonup u
~
\mathrm{in}
~
X,~
u_{n}\rightarrow u
~
\mathrm{a.e. ~in}
~\mathbb{R}^{N},
u_{n}\rightarrow u
~
\mathrm{in}
~
L^{q}_{loc}(\mathbb{R}^{N}).
\end{aligned}
\end{equation*}
Applying Lemma \ref{Lemma3.1} and $\lim\limits_{n\to\infty}\int_{\mathbb{R}^{N}}|u_{n}|^{q}\mathrm{d}x>0$,
there exists
$C>0$
such that 
\begin{equation*}
\begin{aligned}
\sup_{z\in \mathbb{R}^{N}}
\int_{B(z,1)}
|u|^{q}
\mathrm{d}x\geqslant C>0.
\end{aligned}
\end{equation*}
Note that $\{u_{n}\}$ is bounded in $X$ and $X\hookrightarrow L^{q}(\mathbb{R}^{N})$, 
we have that
\begin{equation*}
\begin{aligned}
\sup_{z\in \mathbb{R}^{N}}
\int_{B(z,1)}
|u|^{q}
\mathrm{d}x\leqslant\int_{\mathbb{R}^{N}}
|u_{n}|^{q}
\mathrm{d}x \leqslant C.
\end{aligned}
\end{equation*}
Hence, there exists $C_{0}$ such that 
\begin{equation*}
\begin{aligned}
C_{0}
\leqslant
\sup_{z\in \mathbb{R}^{N}}
\int_{B(z,1)}
|u|^{q}
\mathrm{d}x\leqslant C_{0}^{-1}.
\end{aligned}
\end{equation*}
From the above inequality,
there exists
$z_{n}\in \mathbb{R}^{N}$
such that
\begin{equation*}
\begin{aligned}
\int_{B(z_{n},1)}
|u_{n}|^{q}
\mathrm{d}x
\geqslant
\sup_{z\in \mathbb{R}^{N}}
\int_{B(z,1)}
|u_{n}|^{q}
\mathrm{d}x
-
\frac{C}{2n}
\geqslant
C_{1}>0.
\end{aligned}
\end{equation*}
Set $\bar{u}_{n}:=u_{n}(x+z_{n})$. Then $\|\bar{u}_{n}\|_{X}=\|u_{n}\|_{X}$ and
\begin{equation*}
\begin{aligned}
\int_{B(0,1)}
|\bar{u}_{n}|^{q}
\mathrm{d}x
\geqslant
C_{1}>0.
\end{aligned}
\end{equation*}
Up to a subsequence,
there exists
$\bar{u}$
such that
\begin{equation*}
\begin{aligned}
\bar{u}_{n}\rightharpoonup \bar{u}
\;
\mathrm{in}
~
X,
\;\;
\bar{u}_{n}\rightarrow \bar{u}
~
\mathrm{a.e. ~in}
~
\mathbb{R}^{N}.
\end{aligned}
\end{equation*}
Since the embedding
$X
\hookrightarrow
L^{q}_{loc}(\mathbb{R}^{N})$ is compact,
we deduce that
$\bar{u}\not\equiv0$.
\end{proof}

\subsection{Generalized version of Lieb's translation theorem}

Let 
$q\in[1,\infty)$
and
$\varpi\in(0,N]$,
the usual homogeneous Morrey space is defined by
\begin{equation*}
\begin{aligned}
\mathcal{M}^{q,\varpi}(\mathbb{R}^{N}):=\{u\in L^{1}_{loc}(\mathbb{R}^{N}) |\|u\|^{q}_{\mathcal{M}^{q,\varpi}(\mathbb{R}^{N})}<+\infty\}
\end{aligned}
\end{equation*}
equipped with the norm
\begin{equation*}
\begin{aligned}
\|u\|^{q}_{\mathcal{M}^{q,\varpi}(\mathbb{R}^{N})}
:=
\sup\limits_{R>0,x\in\mathbb{R}^{N}}
R^{\varpi-N}
\int_{B(x,R)}
|u(y)|^{q}
\mathrm{d}y.
\end{aligned}
\end{equation*}
\begin{lemma}
{\rm \cite{Palatucci2014CV}}
\label{Lemma3.2}
For
$N\geqslant3$
and $s\in(0,1]$,
there exists
$C>0$
such that
for
$\iota$
and
$\vartheta$
satisfying
$\frac{2}{2_{s}^{*}}\leqslant\iota<1$ and
$1\leqslant \vartheta<2_{s}^{*}$,
we have
\begin{align*}
\left(
\int_{\mathbb{R}^{N}}
|u|^{2_{s}^{*}}
\mathrm{d}x
\right)^{\frac{1}{2_{s}^{*}}}
\leqslant
C
\|u\|_{D^{s,2}(\mathbb{R}^{N})}^{\iota}
\|u\|_{\mathcal{M}^{\vartheta,\frac{\vartheta(N-2s)}{2}}(\mathbb{R}^{N})}^{1-\iota}
\end{align*}
for
$u\in D^{s,2}(\mathbb{R}^{N})$.
\end{lemma}

By Lemma \ref{Lemma3.2}, we prove a generalized version of Lieb's translation theorem.

\begin{theorem}\label{Theorem3.2}
Suppose that $N\geqslant3$ and $s\in(0,1)$ hold.
Let $\{u_{n}\}\subset X$ be any bounded sequence satisfying:
\begin{equation*}
\begin{aligned}
\lim\limits_{n\to\infty}
\int_{\mathbb{R}^{N}}
|u_{n}|^{2_{s}^{*}}
\mathrm{d}x>0
~{\rm and}~
\lim\limits_{n\to\infty}
\int_{\mathbb{R}^{N}}
|u_{n}|^{2^{*}}
\mathrm{d}x>0.
\end{aligned}
\end{equation*}
Then there exists $\{x_{n}\}\subset \mathbb{R}^{N}$ such that $\{\bar{u}_{n}:=u_{n}(x+x_{n})\}$ converges strongly to $\bar{u}\not\equiv0$ in $L^{2}_{loc}(\mathbb{R}^{N})$.
\end{theorem}

\begin{proof}
We divide this proof into three steps.

{\bf Step 1.}
Since
$\{u_{n}\}\subset X$
is a bounded sequence,
up to a subsequence,
we assume that
\begin{equation*}
\begin{aligned}
u_{n}\rightharpoonup u
~
\mathrm{in}
~
X,~\
u_{n}\rightarrow u
~
\mathrm{a.e. ~in}
~
\mathbb{R}^{N},\ \,
u_{n}\rightarrow u
~
\mathrm{in}
~
L^{2}_{loc}(\mathbb{R}^{N}).
\end{aligned}
\end{equation*}
It follows from Lemma \ref{Lemma3.2} and $\lim\limits_{n\to\infty}
\int_{\mathbb{R}^{N}}
|u_{n}|^{2^{*}}
\mathrm{d}x>0$ that there exists
$C>0$
such that 
\begin{equation*}
\begin{aligned}
\|u_{n}\|_{\mathcal{M}^{2,N-2}(\mathbb{R}^{N})}\geqslant C>0.
\end{aligned}
\end{equation*}

We have the following chain
\begin{equation*}
\begin{aligned}
X
\hookrightarrow
D^{1,2}(\mathbb{R}^{N})
\hookrightarrow L^{2^{*}}(\mathbb{R}^{N})\hookrightarrow \mathcal{M}^{2,N-2}(\mathbb{R}^{N}).
\end{aligned}
\end{equation*}
Then
\begin{equation*}
\begin{aligned}
\|u_{n}\|_{\mathcal{M}^{2,N-2}(\mathbb{R}^{N})}\leqslant C.
\end{aligned}
\end{equation*}
Hence, 
\begin{equation*}
\begin{aligned}
C_0
\leqslant
\|u_{n}\|_{\mathcal{M}^{2,N-2}(\mathbb{R}^{N})}\leqslant C_0^{-1}.
\end{aligned}
\end{equation*}
where $C_{0}$ is independent of $n$.
From the latter inequality,
one deduces that
there exist
$\sigma_{n} > 0$
and
$x_{n}\in \mathbb{R}^{N}$
such that
\begin{equation}\label{3.2}
\begin{aligned}
\sigma_{n}^{-2}
\int_{B(x_{n},\sigma_{n})}
|u_{n}(y)|^{2}
\mathrm{d}y
\geqslant
\|u_{n}\|_{\mathcal{M}^{2,N-2}(\mathbb{R}^{N})}^{2}
-
\frac{C}{2n}
\geqslant
C_{1}>0,
\end{aligned}
\end{equation}
where $C_{1}$ is independent of $n$.

{\bf Step 2.}
First, we show that
$\lim\limits_{n\rightarrow\infty}\sigma_{n}\not=\infty$.
We argue by contradiction, we suppose that
$\lim\limits_{n\rightarrow\infty}\sigma_{n}=\infty.$
Using the boundedness of $\{u_{n}\}$,
we get that
\begin{equation*}
\begin{aligned}
0<
\lim_{n\rightarrow\infty}
\int_{\mathbb{R}^N}
|u_{n}|^{2_{s}^{*}}
\mathrm{d}y
\leqslant C.
\end{aligned}
\end{equation*}
It follows for $N\geqslant3$ and $s\in(0,1)$ that
\begin{equation}\label{3.3}
\begin{aligned}
-2+N(1-\frac{2}{2_{s}^{*}})<0.
\end{aligned}
\end{equation}
Using \eqref{3.2} and \eqref{3.3}, we get that
\begin{equation*}
\begin{aligned}
0<
C_{1}
\leqslant&
\sigma_{n}^{-2}
\int_{B(x_{n},\sigma_{n})}
|u_{n}(y)|^{2}
\mathrm{d}y\\
\leqslant&
\sigma_{n}^{-2}
\left(
\int_{B(0,\sigma_{n})}
\mathrm{d}y
\right)^{\frac{2_{s}^{*}-2}{2_{s}^{*}}}
\left(
\int_{B(x_{n},\sigma_{n})}
|u_{n}(y)|^{2_{s}^{*}}
\mathrm{d}y
\right)^{\frac{2}{2_{s}^{*}}}\\
=&
\sigma_{n}^{-2}
\left(
\frac{\omega_{N-1}}{N}
\sigma_{n}^{N}
\right)^{\frac{2_{s}^{*}-2}{2_{s}^{*}}}
\left(
\int_{B(x_{n},\sigma_{n})}
|u_{n}(y)|^{2_{s}^{*}}
\mathrm{d}y
\right)^{\frac{2}{2_{s}^{*}}}\\
\leqslant&
C
\sigma_{n}^{-2+N(1-\frac{2}{2_{s}^{*}})}
\to 0, \ \text{as}\ n\rightarrow\infty.
\end{aligned}
\end{equation*}
This yields to a contradiction.
By the Bolzano--Weierstrass theorem, up to a subsequence, still denoted by $\{\sigma_{n}\}$,
there exists $\bar{\sigma}\in[0,\infty)$ such that
\begin{equation*}
\begin{aligned}
\lim_{n\rightarrow\infty}\sigma_{n}=\bar{\sigma}.
\end{aligned}
\end{equation*}

Second, we show that $\lim\limits_{n\rightarrow\infty}\sigma_{n}=\bar{\sigma}\not=0$.
Suppose on the contrary that
$\lim\limits_{n\rightarrow\infty}\sigma_{n}=\bar{\sigma}=0$.
Using the boundedness of $\{u_{n}\}$,
we have that
\begin{equation*}
\begin{aligned}
C
\lim_{n\rightarrow\infty}
\left(
\int_{\mathbb{R}^N}
|u_{n}|^{2_{s}^{*}}
\mathrm{d}y
\right)^{\frac{2}{2_{s}^{*}}}
\leqslant
\lim_{n\rightarrow\infty}
\|u_{n}\|_{X}^{2}
\leqslant \bar{C}.
\end{aligned}
\end{equation*}

It follows from H\"{o}lder's and Sobolev's inequalities that
\begin{equation*}
\begin{aligned}
\int_{B(0,\sigma_{n})}
|u_{n}|^{2_{s}^{*}}
\mathrm{d}y
\leqslant&
\left(
\int_{B(0,\sigma_{n})}
\mathrm{d}y
\right)^{\frac{2^{*}-2_{s}^{*}}{2^{*}}}
\left(
\int_{B(0,\sigma_{n})}
|u_{n}|^{2^{*}}
\mathrm{d}y
\right)^{\frac{2_{s}^{*}}{2^{*}}}\\
\leqslant&
S_{1}^{-\frac{2_{s}^{*}}{2}}
\left(
\int_{B(0,\sigma_{n})}
\mathrm{d}y
\right)^{\frac{2^{*}-2_{s}^{*}}{2^{*}}}
\left(
\int_{B(0,\sigma_{n})}
|\nabla u_{n}|^{2}
{\rm d}y
\right)^{\frac{2_{s}^{*}}{2}}\\
\leqslant&
S_{1}^{-\frac{2_{s}^{*}}{2}}
\left(
\int_{B(0,\sigma_{n})}
\mathrm{d}y
\right)^{\frac{2^{*}-2_{s}^{*}}{2^{*}}}
\|u_{n}\|_{D^{1,2}(\mathbb{R}^{N})}^{2_{s}^{*}-2}
\int_{B(0,\sigma_{n})}
|\nabla u_{n}|^{2}
{\rm d}y\\
\leqslant&
C
S_{1}^{-\frac{2_{s}^{*}}{2}}
\left(
\int_{B(0,\sigma_{n})}
\mathrm{d}y
\right)^{\frac{2^{*}-2_{s}^{*}}{2^{*}}}
\int_{B(0,\sigma_{n})}
|\nabla u_{n}|^{2}
{\rm d}y.
\end{aligned}
\end{equation*}
Similarly, for each $z\in \mathbb{R}^N$ we have that
\begin{equation*}
\begin{aligned}
\int_{B(z,\sigma_{n})}
|u_{n}|^{2_{s}^{*}}
\mathrm{d}y
\leqslant
CS_{1}^{-\frac{2_{s}^{*}}{2}}
\left(
\int_{B(0,\sigma_{n})}
\mathrm{d}y
\right)^{\frac{2^{*}-2_{s}^{*}}{2^{*}}}
\int_{B(z,\sigma_{n})}
|\nabla u_{n}|^{2}
{\rm d}y.
\end{aligned}
\end{equation*}
Covering
$\mathbb{R}^N$
with balls of radius
$\sigma_{n}$,
in such a way that each point of
$\mathbb{R}^N$
is contained in at most
$N + 1$
balls, one gets
\begin{equation*}
\begin{aligned}
\int_{\mathbb{R}^N}
|u_{n}|^{2_{s}^{*}}
\mathrm{d}y
\leqslant&
C(N+1)S_{1}^{-\frac{2_{s}^{*}}{2}}
\left(
\int_{B(0,\sigma_{n})}
\mathrm{d}y
\right)^{\frac{2^{*}-2_{s}^{*}}{2^{*}}}
\int_{\mathbb{R}^{N}}
|\nabla u_{n}|^{2}
{\rm d}y\\
\leqslant&
C(N+1)S_{1}^{-\frac{2_{s}^{*}}{2}}
\left(
\int_{B(0,\sigma_{n})}
\mathrm{d}y
\right)^{\frac{2^{*}-2_{s}^{*}}{2^{*}}}
\\
\leqslant&
C(N+1)S_{1}^{-\frac{2_{s}^{*}}{2}}
\sigma_{n}^{N(1-\frac{2_{s}^{*}}{2^{*}})}.
\end{aligned}
\end{equation*}
By our assumption $\lim\limits_{n\rightarrow\infty}\sigma_{n}=0$, it follows that
\begin{equation*}
\begin{aligned}
\lim\limits_{n\rightarrow\infty}
\int_{\mathbb{R}^N}
|u_{n}|^{2_{s}^{*}}
\mathrm{d}y
\leqslant&
C(N+1)S_{1}^{-\frac{2_{s}^{*}}{2}}
\lim\limits_{n\rightarrow\infty}
\sigma_{n}^{N(1-\frac{2_{s}^{*}}{2^{*}})}
=
0,
\end{aligned}
\end{equation*}
which contradicts the facts that $\lim\limits_{n\rightarrow\infty}
\int_{\mathbb{R}^N}
|u_{n}|^{
2_{s}^{*}
}
\mathrm{d}x>0$.

{\bf Step 3.}
Using $\lim\limits_{n\rightarrow\infty}\sigma_{n}=\bar{\sigma}\not=0$,
up to a subsequence,
we have
$\{\sigma_{n}\}\subset (\frac{\bar{\sigma}}{2},2\bar{\sigma})$
and
\begin{equation*}
\begin{aligned}
\frac{2^{2}}{\bar{\sigma}^{2}}
\int_{B(x_{n},2\bar{\sigma})}
|u_{n}(y)|^{2}
\mathrm{d}y
\geqslant
C_{1}>0,
\end{aligned}
\end{equation*}
which gives us:
\begin{equation}\label{3.4}
\begin{aligned}
\int_{B(x_{n},2\bar{\sigma})}
|u_{n}(y)|^{2}
\mathrm{d}y
\geqslant
\frac{C_{1}\bar{\sigma}^{2}}{2^{2s}}>0.
\end{aligned}
\end{equation}
Set $\bar{u}_{n}:=u_{n}(x+x_{n})$. Then $\{u_{n}\}\subset X$ is a bounded sequence satisfying:
\begin{equation*}
\begin{aligned}
\lim\limits_{n\to\infty}
\int_{\mathbb{R}^{N}}
|\bar{u}_{n}|^{2_{s}^{*}}
\mathrm{d}x>0
~{\rm and}~
\lim\limits_{n\to\infty}
\int_{\mathbb{R}^{N}}
|\bar{u}_{n}|^{2^{*}}
\mathrm{d}x>0.
\end{aligned}
\end{equation*}
and from \eqref{3.4},
\begin{equation}\label{3.5}
\begin{aligned}
\int_{B(0,2\bar{\sigma})}
|\bar{u}_{n}(y)|^{2}
\mathrm{d}y
\geqslant
\frac{C_{1}\bar{\sigma}^{2}}{2^{2s}}>0.
\end{aligned}
\end{equation}
Using the combinations of the embedding
$X
\hookrightarrow
D^{1,2}(\mathbb{R}^{N})
\hookrightarrow
L^{2}_{loc}(\mathbb{R}^{N})$ and \eqref{3.5},
we obtain $\bar{u}_{n}\rightharpoonup \bar{u}\not\equiv0$.
\end{proof}

\section{Mountain pass geometry and Nehari manifold}
The energy functional corresponding to the equation \eqref{D} is
\begin{equation*}
\begin{aligned}
I(u)=
\frac{1}{2}
\|u\|_{X}^{2}
-
\frac{1}{2_{s}^{*}}
\int_{\mathbb{R}^{N}}
|u|^{2_{s}^{*}}
\mathrm{d}x
-
\frac{\beta}{p}
\int_{\mathbb{R}^{N}}
|u|^{p}
\mathrm{d}x
-
\frac{1}{2^{*}}
\int_{\mathbb{R}^{N}}
|u|^{2^{*}}
\mathrm{d}x
.
\end{aligned}
\end{equation*}
For any $\phi\in X$, the critical points of $I$ satisfy
\begin{equation*}
\begin{aligned}
0=\langle I'(u),\phi\rangle
=&
\int_{\mathbb{R}^{N}}
\nabla u
\nabla \phi
\mathrm{d}x
+
\int_{\mathbb{R}^{N}}
\int_{\mathbb{R}^{N}}
\frac{(u(x)-u(y))(\phi(x)-\phi(y))}{|x-y|^{N+2s}}
\mathrm{d}x
\mathrm{d}y\\
&
-
\int_{\mathbb{R}^{N}}
|u|^{2_{s}^{*}-1}\phi
\mathrm{d}x
-\beta
\int_{\mathbb{R}^{N}}
|u|^{p-1}\phi
\mathrm{d}x
-
\int_{\mathbb{R}^{N}}
|u|^{2^{*}-1}\phi
\mathrm{d}x
.
\end{aligned}
\end{equation*}

\begin{lemma}\label{Lemma4.1}
Let $N\geqslant3$ and $0<s<1$.
Then the functional $I$ has mountain pass geometric structure.
\end{lemma}
\begin{proof}
Using Lemma \ref{Lemma2.2}, one has
\begin{equation*}
\begin{aligned}
I(u)
\geqslant \frac{1}{2}\|u\|_{X}^{2}
-C\|u\|_{X}^{p}
-C\|u\|_{X}^{2_{s}^{*}}
-C\|u\|_{X}^{2^{*}}.
\end{aligned}
\end{equation*}
Recall that $2<2_{s}^{*}<p<2^{*}$. There exists a sufficiently small positive number $\rho$ such that
\begin{equation*}
\begin{aligned}
\varsigma :=
\inf_{\|u\|_{X}=\rho}
I(u)>0=I(0).
\end{aligned}
\end{equation*}
For $u\in X\setminus\{0\}$, we have
\begin{equation*}
\begin{aligned}
I(tu)
=
\frac{t^2}{2}\|u\|_{X}^{2}
-\beta\frac{t^{p}}{p}\int_{\mathbb{R}^{N}}|u|^{p}\mathrm{d}x
-\frac{t^{2_{s}^{*}}}{2_{s}^{*}}\int_{\mathbb{R}^{N}}|u|^{2_{s}^{*}}\mathrm{d}x
-\frac{t^{2^{*}}}{2^{*}}\int_{\mathbb{R}^{N}}|u|^{2^{*}}\mathrm{d}x.
\end{aligned}
\end{equation*}
From $2<2_{s}^{*}<p<2^{*}$, it follows that $I(tu)<0$ for $t$ large enough.

From the above,
we can choose $t_{0}>0$ corresponding to $u$ such that $I(t_{0}u)<0$ for $t>t_{0}$ and $\|t_{0}u\|_{X}>\rho$.
\end{proof}

We set the mountain pass level
\begin{equation*}
\begin{aligned}
c=\inf\limits_{\gamma\in \Gamma}\sup\limits_{t\in[0, 1]}I(\gamma (t))>0,
\end{aligned}
\end{equation*}
and
\begin{equation*}
\begin{aligned}
\Gamma=\{\gamma\in C\left([0, 1], X\right)| \gamma(0)=0, I(\gamma (1))<0\}.
\end{aligned}
\end{equation*}
We recall the $(PS)_{c}$ sequence as follows.
\begin{definition}
If sequence $\{u_{n}\}\subset X$ satisfies the condition
\begin{equation*}
\begin{aligned}
I(u_{n})\rightarrow c
~\mathrm{and}~I'(u_{n})\rightarrow 0
~\mathrm{in}~X^{-1},
~{\rm as}~n\to\infty.
\end{aligned}
\end{equation*}
Then $\{u_{n}\}$ is called the Palais-Smale sequence of $I$ with respect to $c$, short for $(PS)_{c}$ sequence, 
where $X^{-1}$ is the dual space of $X$.
\end{definition}

It follows from the mountain pass Theorem \cite{Ambrosetti-Rabinowitz1973JFA} and Lemma \ref{Lemma4.1} that there exists a $(PS)_{c}$ sequence.

We now set the Nehari manifold as follows
\begin{equation*}
\begin{aligned}
\mathcal{N}=\{u\in X \setminus\{0\}|\langle I'(u),u\rangle=0\}.
\end{aligned}
\end{equation*}

\begin{lemma}\label{Lemma4.2}
Let $N\geqslant3$ and $0<s<1$.
Then for any $u\in X\setminus\{0\}$, there exists a unique $t_{u}>0$ such that $t_{u}u\in \mathcal{N}$ and $I(t_{u}u)=\max\limits_{t>0}I(tu)$.
\end{lemma}

\begin{proof}
For any $u\in X\setminus\{0\}$ and $t\in (0,\infty)$, we define
\begin{equation*}
\begin{aligned}
f_{1}(t)=I(tu)=&\frac{t^{2}}{2}\|u\|_{X}^{2}
-\beta\frac{t^{p}}{p}\int_{\mathbb{R}^{N}}|u|^{p}\mathrm{d}x
-\frac{t^{2_{s}^{*}}}{2_{s}^{*}}\int_{\mathbb{R}^{N}}|u|^{2_{s}^{*}}\mathrm{d}x
-\frac{t^{2^{*}}}{2^{*}}\int_{\mathbb{R}^{N}}|u|^{2^{*}}\mathrm{d}x.
\end{aligned}
\end{equation*}
We compute
\begin{equation*}
\begin{aligned}
f_{1}'(t)
=&t\|u\|_{X}^{2}
-\beta t^{p-1}\int_{\mathbb{R}^{N}}|u|^{p}\mathrm{d}x
-t^{2_{s}^{*}-1}\int_{\mathbb{R}^{N}}|u|^{2_{s}^{*}}\mathrm{d}x
-t^{2^{*}-1}\int_{\mathbb{R}^{N}}|u|^{2^{*}}\mathrm{d}x.
\end{aligned}
\end{equation*}
We know that $f_{1}'(\cdot)=0$ iff
\begin{equation*}
\begin{aligned}
\|u\|_{X}^{2}
=\beta t^{p-2}\int_{\mathbb{R}^{N}}|u|^{p}\mathrm{d}x
+t^{2_{s}^{*}-2}\int_{\mathbb{R}^{N}}|u|^{2_{s}^{*}}\mathrm{d}x
+t^{2^{*}-2}\int_{\mathbb{R}^{N}}|u|^{2^{*}}\mathrm{d}x.
\end{aligned}
\end{equation*}
Let
\begin{equation*}
\begin{aligned}
f_{2}(t)
=\beta t^{p-2}\int_{\mathbb{R}^{N}}|u|^{p}\mathrm{d}x
+t^{2_{s}^{*}-2}\int_{\mathbb{R}^{N}}|u|^{2_{s}^{*}}\mathrm{d}x
+t^{2^{*}-2}\int_{\mathbb{R}^{N}}|u|^{2^{*}}\mathrm{d}x.
\end{aligned}
\end{equation*}
Clearly, $\lim\limits_{t\rightarrow0}f_{2}(t)\rightarrow0$, $\lim\limits_{t\rightarrow+\infty}f_{2}(t)\rightarrow+\infty$ .
By using the intermediate value theorem, we know that there exists a $0<t_{u}<\infty$ such that
\begin{equation*}
\begin{aligned}
f_{2}(t_{u})=\|u\|_{X}^{2}.
\end{aligned}
\end{equation*}
Furthermore, it is easy to see that $f_{2}(\cdot)$ is strictly increasing on $(0,\infty)$.
Then we get the uniqueness of $t_{u}$.
And then,
\begin{equation*}
\begin{aligned}
\|u\|_{X}^{2}
=\beta t_{u}^{p-2}\int_{\mathbb{R}^{N}}|u|^{p}\mathrm{d}x
+t_{u}^{2_{s}^{*}-2}\int_{\mathbb{R}^{N}}|u|^{2_{s}^{*}}\mathrm{d}x
+t_{u}^{2^{*}-2}\int_{\mathbb{R}^{N}}|u|^{2^{*}}\mathrm{d}x,
\end{aligned}
\end{equation*}
which implies that
\begin{equation*}
\begin{aligned}
\|t_{u}u\|_{X}^{2}
=\beta
\int_{\mathbb{R}^{N}}|t_{u}u|^{p}\mathrm{d}x
+\int_{\mathbb{R}^{N}}|t_{u}u|^{2_{s}^{*}}\mathrm{d}x
+\int_{\mathbb{R}^{N}}|t_{u}u|^{2^{*}}\mathrm{d}x.
\end{aligned}
\end{equation*}
This implies that $t_{2}u\in \mathcal{N}$.
\end{proof}

\begin{lemma}\label{Lemma4.3}
Let $N\geqslant3$ and $0<s<1$.
Then we have
\begin{equation*}
\begin{aligned}
\bar{c}=\inf\limits_{u\in\mathcal{N}}I(u)>0.
\end{aligned}
\end{equation*}
\end{lemma}

\begin{proof}
By applying $\langle I'(u),u\rangle=0$,
we know that
\begin{equation*}
\begin{aligned}
0
=\langle I'(u),u\rangle\geqslant \|u\|_{X}^{2}
-C\|u\|_{X}^{p}
-C\|u\|_{X}^{2_{s}^{*}}
-C\|u\|_{X}^{2^{*}},
\end{aligned}
\end{equation*}
which implies that
\begin{equation*}
\begin{aligned}
C\|u\|_{X}^{p-2}
+C\|u\|_{X}^{2_{s}^{*}-2}
+C\|u\|_{X}^{2^{*}-2}
\geqslant
1,
\end{aligned}
\end{equation*}
and
\begin{equation*}
\begin{aligned}
\|u\|_{X}^{2}
\geqslant C.
\end{aligned}
\end{equation*}
Then,
for $u\in \mathcal{N}$,
we get that
\begin{equation*}\label{36}
\begin{aligned}
I(u)
=&
I(u)-\frac{1}{2_{s}^{*}}\langle I'(u),u\rangle\\
=&\frac{1}{2}\|u\|_{X}^{2}
-\frac{1}{2_{s}^{*}}
\int_{\mathbb{R}^{N}}
|u|^{2_{s}^{*}}
\mathrm{d}x
-
\frac{\beta}{p}
\int_{\mathbb{R}^{N}}
|u|^{p}
\mathrm{d}x
-\frac{1}{2^{*}}
\int_{\mathbb{R}^{N}}
|u|^{2^{*}}
\mathrm{d}x\\
&-\frac{1}{2_{s}^{*}}
\left(\|u\|_{X}^{2}
-\int_{\mathbb{R}^{N}}
|u|^{2_{s}^{*}}
\mathrm{d}x
-\beta
\int_{\mathbb{R}^{N}}
|u|^{p}
\mathrm{d}x
-\int_{\mathbb{R}^{N}}
|u|^{2^{*}}
\mathrm{d}x
\right)\\
=&
\left(\frac{1}{2}-\frac{1}{ 2_{s}^{*}}\right)
\|u\|_{X}^{2}
+\beta
\left(\frac{1}{2_{s}^{*}}-\frac{1}{p}\right)
\int_{\mathbb{R}^{N}}
|u|^{p}
\mathrm{d}x
+
\left(\frac{1}{2_{s}^{*}}-\frac{1}{2^{*}}\right)
\int_{\mathbb{R}^{N}}
|u|^{2^{*}}
\mathrm{d}x\\
\geqslant&
\left(\frac{1}{2}-\frac{1}{ 2_{s}^{*}}\right)
\|u\|_{X}^{2}
\geqslant
C.
\end{aligned}
\end{equation*}
Hence, $I$ is bounded from below on $\mathcal{N}$. And then $\bar{c}>0$.
\end{proof}

Set
\begin{equation*}
\begin{aligned}
\bar{\bar{c}}:=\inf\limits_{u\in X\setminus\{0\}}\sup\limits_{t\geqslant 0}I(tu).
\end{aligned}
\end{equation*}
From \cite{Willem1996Book}, we have the following lemma.
\begin{lemma}\label{Lemma4.4}
Let $N\geqslant3$ and $0<s<1$.
Then we have
$c=\bar{c}=\bar{\bar{c}}$.
\end{lemma}

\begin{proof}
By virtue of Lemma \ref{Lemma4.2}, we have the following result directly
\begin{equation*}
\begin{aligned}
\bar{c}=\bar{\bar{c}}.
\end{aligned}
\end{equation*}
For any $u\in X\setminus\{0\}$,
there exists some $\tilde{t}>0$ large,
such that $I(\tilde{t}u)<0$.
Define a path $\gamma:[0, 1]\to X$ by $\gamma(t) = t\tilde{t}u$.
Clearly,
$\gamma\in\Gamma$ and
\begin{equation*}
\begin{aligned}
c\leqslant \bar{\bar{c}}.
\end{aligned}
\end{equation*}
On the other hand,
for each path $\gamma\in\Gamma$,
let $g(t):=\langle I'(\gamma(t)),\gamma(t)\rangle$.
Then,
$g(0)=0$ and $g(t)>0$ for $t>0$ small.
A direct calculation gives us
\begin{equation*}
\begin{aligned}
I(\gamma(1))
-\frac{1}{2_{s}^{*}}\langle I'(\gamma(1)),\gamma(1)\rangle
\geqslant \left(\frac{1}{2}-\frac{1}{2_{s}^{*}}\right)\|\gamma(1)\|_{X}^{2}
\geqslant 0,
\end{aligned}
\end{equation*}
which implies that
\begin{equation*}
\begin{aligned}
\langle I'(\gamma(1)),\gamma(1)\rangle
\leqslant 2_{s}^{*}\cdot I(\gamma(1))
=2_{s}^{*}\cdot I(\tilde{t}u)<0.
\end{aligned}
\end{equation*}
Thus, there exists $\tilde{\tilde{t}} \in(0,1)$ such that $g(\tilde{\tilde{t}}) = 0$, i.e. $\gamma(\tilde{\tilde{t}})\in \mathcal{N}$ and $c\geqslant \bar{c}$.
This shows that
$c=\bar{c}=\bar{\bar{c}}$.
\end{proof}

\begin{lemma}\label{Lemma4.5}
Let $N\geqslant3$ and $0<s<1$.
For $u\in \mathcal{N}$, we have $\Phi'(u)\not=0$, where
\begin{equation}\label{4.1}
\begin{aligned}
\Phi(u)=\langle I'(u),u\rangle=\|u\|_{X}^{2}
-\int_{\mathbb{R}^{N}}|u|^{2_{s}^{*}}\mathrm{d}x
-\beta\int_{\mathbb{R}^{N}}|u|^{p}\mathrm{d}x
-\int_{\mathbb{R}^{N}}|u|^{2^{*}}\mathrm{d}x,
\end{aligned}
\end{equation}
and
\begin{equation}\label{4.2}
\begin{aligned}
\langle\Phi'(u),u\rangle
=&2\|u\|_{X}^{2}
-2_{s}^{*}\int_{\mathbb{R}^{N}}|u|^{2_{s}^{*}}\mathrm{d}x
-p\beta\int_{\mathbb{R}^{N}}|u|^{p}\mathrm{d}x
-2^{*}\int_{\mathbb{R}^{N}}|u|^{2^{*}}\mathrm{d}x.
\end{aligned}
\end{equation}
Moreover, if $\bar{u}\in \mathcal{N}$ and $I(\bar{u})=c$, then $u$ is a ground state solution for equation \eqref{D}.
\end{lemma}

\begin{proof}
For $u\in \mathcal{N}$, it follows from \eqref{4.1} and \eqref{4.2} that
\begin{equation*}
\begin{aligned}
\langle\Phi'(u),u\rangle
=&\langle\Phi'(u),u\rangle-2_{s}^{*}\Phi(u)\\
=&2\|u\|_{X}^{2}
-2_{s}^{*}\int_{\mathbb{R}^{N}}|u|^{2_{s}^{*}}\mathrm{d}x
-p\beta\int_{\mathbb{R}^{N}}|u|^{p}\mathrm{d}x
-2^{*}\int_{\mathbb{R}^{N}}|u|^{2^{*}}\mathrm{d}x\\
&-2_{s}^{*}\left(
\|u\|_{X}^{2}
-\int_{\mathbb{R}^{N}}|u|^{2_{s}^{*}}\mathrm{d}x
-\beta\int_{\mathbb{R}^{N}}|u|^{p}\mathrm{d}x
-\int_{\mathbb{R}^{N}}|u|^{2^{*}}\mathrm{d}x\right)\\
=&(2-2_{s}^{*})\|u\|_{X}^{2}
+(2_{s}^{*}-p)\beta\int_{\mathbb{R}^{N}}|u|^{p}\mathrm{d}x
+(2_{s}^{*}-2^{*})\int_{\mathbb{R}^{N}}|u|^{2^{*}}\mathrm{d}x\\
\leqslant&
(2-2_{s}^{*})\|u\|_{X}^{2}<0.
\end{aligned}
\end{equation*}
Thus, $\Phi'(u)\not=0$ for $u\in \mathcal{N}$.

If $u\in \mathcal{N}$ and $I(u)=\bar{c}$,
since $\bar{c}$ is the minimum of $I$ on $\mathcal{N}$,
by using the Lagrange multiplier theorem, we know that there exists $\lambda\in \mathbb{R}$ such that $I'(u)=\lambda \Phi'(u)$. So 
\begin{equation*}
\begin{aligned}
\langle \lambda \Phi'(u),u\rangle =\langle I'(u),u\rangle=\Phi(u)=0.
\end{aligned}
\end{equation*}
This shows that $\lambda=0$ and $I'(u)=0$.
Thus,
$u$ is a ground state solution for equation \eqref{D}.
\end{proof}

\section{Proof of Theorem \ref{Theorem1.2}}
\begin{lemma}\label{Lemma5.1}
Let $N\geqslant3$ and $0<s<1$.
Then there exists a bounded $(PS)_{c}$ sequence $\{u_{n}\}\subset \mathcal{N}$ such that
\begin{equation*}
\begin{aligned}
I(u_{n})\rightarrow c
\ \,\mathrm{and}\ \,
\|I(u_n)\|_{X^{-1}}\rightarrow 0,\ \mathrm{as}\ n\rightarrow \infty.
\end{aligned}
\end{equation*}
\end{lemma}

\begin{proof}
From Lemma \ref{Lemma4.2}, we know that $\mathcal{N}\not=\emptyset$ and $\inf\limits_{u\in \mathcal{N}}I(u)=\bar{c}=c$.
By the Ekeland's variational principle, there exist $\{u_{n}\}\subset \mathcal{N}$ and
$\lambda_{n}\in \mathbb{R}$ such that
\begin{equation*}
\begin{aligned}
I(u_{n})\rightarrow \bar{c}
~\mathrm{and}~I'(u_{n})-\lambda_{n}\Phi'(\bar{u}_{n})\rightarrow 0
~\mathrm{in}~X^{-1},
~{\rm as}~n\to\infty.
\end{aligned}
\end{equation*}
So
\begin{equation*}
\begin{aligned}
\bar{c}=I(u_{n})
=&J(u_{n})
-\frac{1}{2_{s}^{*}}\langle I'(u_{n}),u_{n}\rangle\\
\geqslant&\left(\frac{1}{2}-\frac{1}{2_{s}^{*}}\right)
\|u_{n}\|_{X}^{2}
+\beta\left(\frac{1}{2_{s}^{*}}-\frac{1}{p}\right)
\int_{\mathbb{R}^{N}}|u_{n}|^{p}\mathrm{d}x
+\left(\frac{1}{2_{s}^{*}}-\frac{1}{2^{*}}\right)
\int_{\mathbb{R}^{N}}|u_{n}|^{2^{*}}\mathrm{d}x,
\end{aligned}
\end{equation*}
which implies that $\{u_{n}\}$ is bounded in $X$.

Taking $n\to\infty$, we have that
\begin{equation*}
\begin{aligned}
|\langle I'(u_{n}),u_{n}\rangle-\langle\lambda_{n}\Phi'(u_{n}),u_{n}\rangle|
\leqslant \|I'(u_{n})
-\lambda_{n}\Phi'(u_{n})\|_{X^{-1}}\|u_{n}\|_{X}
\rightarrow 0,
\end{aligned}
\end{equation*}
which gives us
\begin{equation}\label{5.1}
\begin{aligned}
\langle I'(u_{n}),u_{n}\rangle
-\lambda_{n}\langle \Phi'(u_{n}),u_{n}\rangle\rightarrow 0,
~{\rm as}~n\to\infty.
\end{aligned}
\end{equation}
Note that $\{u_{n}\}\subset \mathcal{N}$. From Lemma \ref{Lemma4.5}, we have
\begin{equation}\label{5.2}
\begin{aligned}
\langle I'(u_{n}),u_{n}\rangle=0,
\end{aligned}
\end{equation}
and
\begin{equation}\label{5.3}
\begin{aligned}
\langle \Phi'(u_{n}),u_{n}\rangle\not=0.
\end{aligned}
\end{equation}
Combining \eqref{5.1}-\eqref{5.3},
we conclude  $\lambda_{n}\rightarrow0$.

By virtue of H\"{o}lder's and Sobolev's inequalities,
we know that
\begin{equation*}
\begin{aligned}
&\|I'(u_{n})\|_{X^{-1}}\\
=&\sup_{\varphi\in X,\|\varphi\|_{X}=1}
|\langle \Phi'(u_{n}),\varphi\rangle|\\
=&\sup_{\varphi\in X,\|\varphi\|_{X}=1}
\left|
2\int_{\mathbb{R}^{N}}
\nabla u \nabla \varphi
\mathrm{d}x
+2
\int_{\mathbb{R}^{N}}
\int_{\mathbb{R}^{N}}
\frac{(u_{n}(x)-u_{n}(y))(\varphi(x)-\varphi(y))}{|x-y|^{N+2s}}
\mathrm{d}x
\mathrm{d}y
\right.\\
&\left.
-2_{s}^{*}
\int_{\mathbb{R}^{N}}
|u_{n}|^{2_{s}^{*}-2}u_{n}\varphi
\mathrm{d}x
-p\beta
\int_{\mathbb{R}^{N}}
|u_{n}|^{p-2}u_{n}\varphi
\mathrm{d}x
-2^{*}
\int_{\mathbb{R}^{N}}
|u_{n}|^{2^{*}-2}u_{n}\varphi
\mathrm{d}x
\right|\\
\leqslant&C.
\end{aligned}
\end{equation*}
Then
\begin{equation*}
\begin{aligned}
\|I'(u_{n})\|_{X^{-1}}
\leqslant\|I'(u_{n})-\lambda_{n}\Phi'(u_{n})\|_{X^{-1}}
+|\lambda_{n}|\|\Phi'(u_{n})\|_{X^{-1}}
=o(1),
\end{aligned}
\end{equation*}
which implies $I'(u_{n})\rightarrow0$ in $X^{-1}$.
Hence, $\{u_{n}\}$ is a $(PS)_{{c}}$ sequence of $I$.
\end{proof}

\begin{lemma}\label{Lemma5.2}
Under the assumptions of Theorem \ref{Theorem1.2} hold.
There exists $\bar{\beta
}\in(0,+\infty)$ such that for any $\beta>\bar{\beta}$, we have
\begin{equation*}
\begin{aligned}
c\in(0,c^{*}),
\end{aligned}
\end{equation*}
where
\begin{equation*}
\begin{aligned}
c^{*}:=\min\left\{\left(\frac{1}{2}-\frac{1}{2_{s}^{*}}\right)S_{s}^{\frac{2_{s}^{*}}{2_{s}^{*}-2}},\left(\frac{1}{2}-\frac{1}{2^{*}}\right)S_{1}^{\frac{2^{*}}{2^{*}-2}}\right\}
\end{aligned}
\end{equation*}
where $S_{1}$ is the best constant of the Sobolev inequality, see Lemma \ref{Lemma2.3}.
\end{lemma}

\begin{proof}
We choose $w\in X$ as follows
\begin{equation*}
\begin{aligned}
\|w\|_{X}=1~~\mathrm{and}~~\int_{\mathbb{R}^{N}}
|w|^{p}
\mathrm{d}x>0.
\end{aligned}
\end{equation*}
From the mountain pass geometric structure,
one has that
\begin{equation*}
\begin{aligned}
\lim_{t\to+\infty} I(tw)=-\infty,
\end{aligned}
\end{equation*}
and $t_{w,\beta}>0$ such that $t_{w,\beta} w\in \mathcal{N}$ 
\begin{equation*}
\begin{aligned}
\sup_{t\geqslant 0} I(tw)=I(t_{w,\beta}w).
\end{aligned}
\end{equation*}
Thus, $t_{w,\beta}$ satisfies
\begin{equation}\label{5.4}
\begin{aligned}
t_{w,\beta}^{2}\|w\|_{X}^{2}=
t_{w,\beta}^{2^{*}}
\int_{\mathbb{R}^{N}}
|w|^{2^{*}}
\mathrm{d}x
+\beta
t_{w,\beta}^{p}
\int_{\mathbb{R}^{N}}
|w|^{p}
\mathrm{d}x
+t_{w,\beta}^{2_{s}^{*}}
\int_{\mathbb{R}^{N}}
|w|^{2_{s}^{*}}
\mathrm{d}x.
\end{aligned}
\end{equation}
Furthermore, 
\begin{equation*}
\begin{aligned}
t_{w,\beta}^{2}\|w\|_{X}^{2}\geqslant 
t_{w,\beta}^{2^{*}}
\int_{\mathbb{R}^{N}}
|w|^{2^{*}}
\mathrm{d}x.
\end{aligned}
\end{equation*}
This implies that $\{t_{w,\beta}\}_{\beta}$ is bounded.

We claim that $t_{w,\beta}\to0$ as $\beta\to+\infty$. Arguing by contradiction, we can assume that there exist $t_{1}>0$ and a sequence $\{\beta_{n}\}$ with $\beta_{n}\to\infty$, such that $t_{w,\beta_{n}}\to t_{1}$ as $n\to+\infty$.
One has
\begin{equation*}
\begin{aligned}
\beta_{n}
t_{w,\beta_{n}}^{p}
\int_{\mathbb{R}^{N}}
|w|^{p}
\mathrm{d}x \to +\infty,~~\mathrm{as}~~n\to+\infty.
\end{aligned}
\end{equation*}
Putting this into \eqref{5.4},
we know that
\begin{equation*}
\begin{aligned}
t_{1}^{2}\|w\|_{E}^{2}=+\infty.
\end{aligned}
\end{equation*}
This is a contradiction with $\|w\|_{E}=1$.

By applying $t_{w,\beta}\to0$ as $\beta\to+\infty$,
we obtain that
\begin{equation*}
\begin{aligned}
\lim_{\beta\to+\infty}\sup_{t\geqslant 0} I(tw)
=\lim_{\beta\to+\infty}
I(t_{w,\beta}w)
=0.
\end{aligned}
\end{equation*}
Then there exists  $\bar{\beta}\in(0,+\infty)$ such that for any $\beta>\bar{\beta}$   there holds
\begin{equation*}
\begin{aligned}
\sup_{t\geqslant 0} I(tw)<c^{*}.
\end{aligned}
\end{equation*}
For any $\beta>\bar{\beta}$,
we construct a mountain pass path as: taking $e=Tw$ and $\gamma(t)=te$ with $T$ large enough to satisfies $J(e)<0$, then
\begin{equation*}
\begin{aligned}
c\leqslant \max_{t\in[0,1]}
I(\gamma(t)).
\end{aligned}
\end{equation*}
Hence,  $c\leqslant \sup\limits_{t\geqslant 0} I(tw)<c^{*}$.
\end{proof}

\begin{lemma}\label{Lemma5.3}
Let $N\geqslant3$, $0<s<1$ and $\{u_{n}\}\subset \mathcal{M}$ be a bounded $(PS)_{c}$ sequence. Then there exists $y_{n}\subset \mathbb{R}^{N}$ such that  $\bar{u}_{n}:=u_{n}(x+y_{n})$ to $u\not \equiv0$ in $X$.
Moreover, $J(u)=c$. 
\end{lemma}
\begin{proof}
From Lemma \ref{Lemma5.1}, we know that there exists a bounded $(PS)_{c}$ sequence $\{u_{n}\}\subset \mathcal{N}$ at level  $c\in(0,c^{*})$.
If $\lim\limits_{n\to\infty}\int_{\mathbb{R}^{N}}|u_{n}|^{2^{*}}\mathrm{d}x=0$, then 
\begin{equation*}
\begin{aligned}
c=J(u_{n})
=
\frac{1}{2}\|u_{n}\|_{X}^{2}-\frac{1}{2_{s}^{*}}
\int_{\mathbb{R}^{N}}
|u_{n}|^{2_{s}^{*}}
\mathrm{d}x,
\end{aligned}
\end{equation*}
and
\begin{equation}\label{5.5}
\begin{aligned}
0=\langle J'(u_{n}),u_{n}\rangle
=\|u_{n}\|_{X}^{2}-
\int_{\mathbb{R}^{N}}
|u_{n}|^{2_{s}^{*}}
\mathrm{d}x,
\end{aligned}
\end{equation}
which implies that
\begin{equation}\label{5.6}
\begin{aligned}
c=\left(\frac{1}{2}-\frac{1}{2_{s}^{*}}\right)
\|u_{n}\|_{X}^{2}.
\end{aligned}
\end{equation}
It follows from \eqref{5.5} and Lemma \ref{Lemma2.3} that
\begin{equation*}
\begin{aligned}
\|u_{n}\|_{X}^{2}=
\int_{\mathbb{R}^{N}}
|u|^{2_{s}^{*}}
\mathrm{d}x
\leqslant
S_{s}^{-\frac{2_{s}^{*}}{2}}
\|u_{n}\|_{D^{s,2}(\mathbb{R}^{N})}^{2_{s}^{*}}
\leqslant
S_{s}^{-\frac{2_{s}^{*}}{2}}
\|u_{n}\|_{X}^{2_{s}^{*}},
\end{aligned}
\end{equation*}
which shows
\begin{equation}\label{5.7}
\begin{aligned}
S_{s}^{\frac{2_{s}^{*}}{2}}
\leqslant
\|u_{n}\|_{X}^{2_{s}^{*}-2}
\Rightarrow
\|u_{n}\|_{X}^{2}\geqslant S_{s}^{\frac{2_{s}^{*}}{2_{s}^{*}-2}}.
\end{aligned}
\end{equation}
Combining \eqref{5.6} and \eqref{5.7}, 
\begin{equation*}
\begin{aligned}
c\geqslant \left(\frac{1}{2}-\frac{1}{2_{s}^{*}}\right)S_{s}^{\frac{2_{s}^{*}}{2_{s}^{*}-2}}.
\end{aligned}
\end{equation*}
Therefore, we get $\lim\limits_{n\to\infty}\int_{\mathbb{R}^{N}}|u_{n}|^{2^{*}}\mathrm{d}x>0$.

If $\lim\limits_{n\to\infty}\int_{\mathbb{R}^{N}}|u_{n}|^{2_{s}^{*}}\mathrm{d}x=0$, then 
\begin{equation*}
\begin{aligned}
c=J(u_{n})
=
\frac{1}{2}\|u_{n}\|_{X}^{2}-\frac{1}{2^{*}}
\int_{\mathbb{R}^{N}}
|u_{n}|^{2^{*}}
\mathrm{d}x,
\end{aligned}
\end{equation*}
and
\begin{equation}\label{5.8}
\begin{aligned}
0=\langle J'(u_{n}),u_{n}\rangle
=\|u_{n}\|_{X}^{2}-
\int_{\mathbb{R}^{N}}
|u_{n}|^{2^{*}}
\mathrm{d}x,
\end{aligned}
\end{equation}
which gives
\begin{equation}\label{5.9}
\begin{aligned}
c=\left(\frac{1}{2}-\frac{1}{2^{*}}\right)
\|u_{n}\|_{X}^{2}.
\end{aligned}
\end{equation}
It follows from \eqref{5.8} and Lemma \ref{Lemma2.3} that
\begin{equation*}
\begin{aligned}
\|u_{n}\|_{X}^{2}=
\int_{\mathbb{R}^{N}}
|u|^{2^{*}}
\mathrm{d}x
\leqslant
S_{1}^{-\frac{2^{*}}{2}}
\|u_{n}\|_{D^{1,2}(\mathbb{R}^{N})}^{2^{*}}
\leqslant
S_{1}^{-\frac{2^{*}}{2}}
\|u_{n}\|_{X}^{2^{*}},
\end{aligned}
\end{equation*}
which shows
\begin{equation}\label{5.10}
\begin{aligned}
\|u_{n}\|_{X}^{2}\geqslant S_{1}^{\frac{2^{*}}{2^{*}-2}}.
\end{aligned}
\end{equation}
Combining \eqref{5.9} and \eqref{5.10}, 
\begin{equation*}
\begin{aligned}
c\geqslant \left(\frac{1}{2}-\frac{1}{2^{*}}\right)S_{1}^{\frac{2^{*}}{2^{*}-2}}.
\end{aligned}
\end{equation*}
Therefore, we get $\lim\limits_{n\to\infty}\int_{\mathbb{R}^{N}}|u_{n}|^{2_{s}^{*}}\mathrm{d}x>0$.

By using Theorem \ref{Theorem3.2}, there exists $y_{n}\subset \mathbb{R}^{N}$ such that  $\bar{u}_{n}:=u_{n}(x+y_{n})\rightharpoonup u\not \equiv0$ in $L^{2}_{loc}(\mathbb{R}^{N})$ and
\begin{equation*}
\begin{aligned}
c
=J(\bar{u}_{n}),~~\mathrm{and}~~0
=\langle J'(\bar{u}_{n}),\varphi\rangle
=\langle J'(u),\varphi\rangle.
\end{aligned}
\end{equation*}
Using Br\'{e}zis-Lieb Lemma \cite{Brezis-Lieb1983PAMS},
one deduces that
\begin{equation*}
\begin{aligned}
\bar{c}
\leqslant J(u)
=&
J(u)-\frac{1}{2_{s}^{*}}\langle J'(u),u\rangle\\
=&\frac{1}{2}\|u\|_{X}^{2}
-\frac{1}{2_{s}^{*}}
\int_{\mathbb{R}^{N}}
|u|^{2_{s}^{*}}
\mathrm{d}x
-
\frac{\beta}{p}
\int_{\mathbb{R}^{N}}
|u|^{p}
\mathrm{d}x
-\frac{1}{2^{*}}
\int_{\mathbb{R}^{N}}
|u|^{2^{*}}
\mathrm{d}x\\
&-\frac{1}{2_{s}^{*}}
\left(\|u\|_{X}^{2}
-\int_{\mathbb{R}^{N}}
|u|^{2_{s}^{*}}
\mathrm{d}x
-\int_{\mathbb{R}^{N}}
|u|^{2^{*}}
\mathrm{d}x
-\beta
\int_{\mathbb{R}^{N}}
|u|^{p}
\mathrm{d}x
\right)\\
=&
\left(\frac{1}{2}-\frac{1}{2_{s}^{*}}\right)
\|u\|_{X}^{2}
+
\left(\frac{1}{2_{s}^{*}}
-\frac{1}{p}
\right)
\int_{\mathbb{R}^{N}}
|u|^{p}
\mathrm{d}x
+
\left(\frac{1}{2_{s}^{*}}
-\frac{1}{2^{*}}
\right)
\int_{\mathbb{R}^{N}}
|u|^{2^{*}}
\mathrm{d}x\\
\leqslant&
\lim\limits_{n\to\infty}
\left[\left(\frac{1}{2}-\frac{1}{2_{s}^{*}}\right)
\|\bar{u}_{n}\|_{X}^{2}
+
\left(\frac{1}{2_{s}^{*}}
-\frac{1}{p}
\right)
\int_{\mathbb{R}^{N}}
|\bar{u}_{n}|^{p}
\mathrm{d}x
+
\left(\frac{1}{2_{s}^{*}}
-\frac{1}{2^{*}}
\right)
\int_{\mathbb{R}^{N}}
|\bar{u}_{n}|^{2^{*}}
\mathrm{d}x
\right]\\
=&
\lim\limits_{n\to\infty}
J(\bar{u}_{n})
-\frac{1}{p}
\lim\limits_{n\to\infty}
\langle J'(\bar{u}_{n}),\bar{u}_{n}\rangle\\
=&
\lim\limits_{n\to\infty}
J(\bar{u}_{n})\\
=&c
=\bar{c}.
\end{aligned}
\end{equation*}
which gives $\lim\limits_{n\to\infty}
\|\bar{u}_{n}\|_{X}^{2}=
\|u\|_{X}^{2}$ and
$J(u)=\bar{c}=c$.
\end{proof}

We are now in a position to prove Theorem \ref{Theorem1.2}.
\begin{proof}[Proof of Theorem \ref{Theorem1.2}]
{\bf Step 1}. From \textcolor{red}{Lemmas \ref{Lemma5.3} and \ref{Lemma4.5}}, we show that equation \eqref{D} has a  ground state solution.
We now show equation \eqref{D} has a nonnegative ground state solution.

For every $u\in\mathcal{N}$, 
it is easy to see that
\begin{equation*}
\begin{aligned}
&\int_{u(y)\geqslant0}
\int_{u(x)<0}
\frac{||u(x)|-|u(y)||^{2}}{|x-y|^{N+2s}}
\mathrm{d}x
\mathrm{d}y
+
\int_{u(y)\geqslant0}
\int_{u(x)<0}
\frac{||u(x)|-|u(y)||^{2}}{|x-y|^{N+2s}}
\mathrm{d}x
\mathrm{d}y\\
&\leqslant
\int_{u(y)\geqslant0}
\int_{u(x)<0}
\frac{|u(x)-u(y)|^{2}}{|x-y|^{N+2s}}
\mathrm{d}x
\mathrm{d}y
+
\int_{u(y)\geqslant0}
\int_{u(x)<0}
\frac{|u(x)-u(y)|^{2}}{|x-y|^{N+2s}}
\mathrm{d}x
\mathrm{d}y,
\end{aligned}
\end{equation*}
which implies 
\begin{equation*}
\begin{aligned}
\|~|u|~\|_{D^{s,2}(\mathbb{R}^{N})}
\leqslant
\|u\|_{D^{s,2}(\mathbb{R}^{N})}.
\end{aligned}
\end{equation*}
Then, 
\begin{equation*}
\begin{aligned}
J(t|u|)
\leqslant
J(tu),
~~t>0.
\end{aligned}
\end{equation*}
Note that $u\neq0$.
Then there exists $t_{1,u}>0$
such that $t_{1,u}|u|\in\mathcal{N}$. And
\begin{equation*}
\begin{aligned}
\|~|u|~\|^{2}_{D^{1,2}(\mathbb{R}^{N})}
+
\|~|u|~\|^{2}_{D^{s,2}(\mathbb{R}^{N})}
=
t_{1,u}^{2_{s}^{*}-2}
\int_{\mathbb{R}^{N}}
|u|^{2_{s}^{*}}
\mathrm{d}x
+t_{1,u}^{p-2}
\beta
\int_{\mathbb{R}^{N}}
|u|^{p}
\mathrm{d}x
+t_{1,u}^{2^{*}-2}
\int_{\mathbb{R}^{N}}
|u|^{2^{*}}
\mathrm{d}x.
\end{aligned}
\end{equation*}
It follows from $u\in\mathcal{N}$ that
\begin{equation*}
\begin{aligned}
&\int_{\mathbb{R}^{N}}
|u|^{2_{s}^{*}}
\mathrm{d}x
+\beta
\int_{\mathbb{R}^{N}}
|u|^{p}
\mathrm{d}x
+\int_{\mathbb{R}^{N}}
|u|^{2^{*}}
\mathrm{d}x\\
=&
\|u\|^{2}_{D^{1,2}(\mathbb{R}^{N})}
+
\|u\|^{2}_{D^{s,2}(\mathbb{R}^{N})}\\
\geqslant&
\|~|u|~\|^{2}_{D^{1,2}(\mathbb{R}^{N})}
+
\|~|u|~\|^{2}_{D^{s,2}(\mathbb{R}^{N})}\\
=&
t_{1,u}^{2_{s}^{*}-2}
\int_{\mathbb{R}^{N}}
|u|^{2_{s}^{*}}
\mathrm{d}x
+t_{1,u}^{p-2}
\beta
\int_{\mathbb{R}^{N}}
|u|^{p}
\mathrm{d}x
+t_{1,u}^{2^{*}-2}
\int_{\mathbb{R}^{N}}
|u|^{2^{*}}
\mathrm{d}x,
\end{aligned}
\end{equation*}
which gives 
\begin{equation*}
\begin{aligned}
t_{1,u}\in(0,1].
\end{aligned}
\end{equation*}
We further have
\begin{equation*}
\begin{aligned}
\bar{c}_{*}
=
J(t_{1,u}|u|)
\leqslant
J(t_{1,u}u)
\leqslant
\max_{t\geqslant 0}
J(tu)
=J(u)
=\bar{c}_{*}.
\end{aligned}
\end{equation*}
Then we know
$J(t_{1,u}|u|)=c_{*}$.
From Lemma \ref{Lemma4.5}, $t_{1,u}|u|$ is a nonnegative ground state solution of equation
\eqref{D}.

{\bf Step 2.} 
In this step, we show that equation \eqref{D} has a radial ground state solution.
We set $v^{*}$ to be the symmetric decreasing rearrangement of $v:=t_{1,u}|u|$.
By the rearrangement inequalities  \cite{Burchard-Hajaiej2006JFA,Hajaiej-Stuart2003PLMS,Hajaiej2005Edinburgh},
we hvae that
\begin{equation*}
\begin{aligned}
\|v^{*}\|_{D^{1,2}(\mathbb{R}^{N})}
\leqslant
\|v\|_{D^{1,2}(\mathbb{R}^{N})},
\end{aligned}
\end{equation*}
\begin{equation*}
\begin{aligned}
\|v^{*}\|_{D^{s,2}(\mathbb{R}^{N})}
\leqslant
\|v\|_{D^{s,2}(\mathbb{R}^{N})},
\end{aligned}
\end{equation*}
\begin{equation*}
\begin{aligned}
\int_{\mathbb{R}^{N}}
|v^{*}|^{2_{s}^{*}}
\mathrm{d}x
=\int_{\mathbb{R}^{N}}
|v|^{2_{s}^{*}}
\mathrm{d}x,
\end{aligned}
\end{equation*}
\begin{equation*}
\begin{aligned}
\int_{\mathbb{R}^{N}}
|v^{*}|^{p}
\mathrm{d}x
=\int_{\mathbb{R}^{N}}
|v|^{p}
\mathrm{d}x,
\end{aligned}
\end{equation*}
\begin{equation*}
\begin{aligned}
\int_{\mathbb{R}^{N}}
|v^{*}|^{2^{*}}
\mathrm{d}x
=\int_{\mathbb{R}^{N}}
|v|^{2^{*}}
\mathrm{d}x.
\end{aligned}
\end{equation*}
This implies:
\begin{equation*}
\begin{aligned}
I_{2^{*}}(t|v^{*}|)
\leqslant
I_{2^{*}}(tv),
~~t>0.
\end{aligned}
\end{equation*}
Notice that $v\not\equiv0$. Then there exists $t_{1,v^{*}}>0$
such that $t_{1,v^{*}}v^{*}\in\mathcal{M}$. And
\begin{equation*}
\begin{aligned}
\|v^{*}\|^{2}_{D^{1,2}(\mathbb{R}^{N})}
+
\|v^{*}\|^{2}_{D^{s,2}(\mathbb{R}^{N})}
=
t_{1,v^{*}}^{2_{s}^{*}-2}
\int_{\mathbb{R}^{N}}
|v^{*}|^{2_{s}^{*}}
\mathrm{d}x
+
t_{1,v^{*}}^{p-2}
\beta
\int_{\mathbb{R}^{N}}
|v^{*}|^{p}
\mathrm{d}x
+t_{1,v^{*}}^{2^{*}-2}
\int_{\mathbb{R}^{N}}
|v^{*}|^{2^{*}}
\mathrm{d}x.
\end{aligned}
\end{equation*}
It follows from $v\in\mathcal{N}$ that
\begin{equation*}
\begin{aligned}
&\int_{\mathbb{R}^{N}}
|v^{*}|^{2_{s}^{*}}
\mathrm{d}x
+\beta
\int_{\mathbb{R}^{N}}
|v^{*}|^{p}
\mathrm{d}x
+\int_{\mathbb{R}^{N}}
|v^{*}|^{2^{*}}
\mathrm{d}x\\
=&\int_{\mathbb{R}^{N}}
|v|^{2_{s}^{*}}
\mathrm{d}x
+\beta
\int_{\mathbb{R}^{N}}
|v|^{p}
\mathrm{d}x
+\int_{\mathbb{R}^{N}}
|v|^{2^{*}}
\mathrm{d}x\\
=&\|v\|^{2}_{D^{1,2}(\mathbb{R}^{N})}
+
\|v\|^{2}_{D^{s,2}(\mathbb{R}^{N})}\\
\geqslant&
\|v^{*}\|^{2}_{D^{1,2}(\mathbb{R}^{N})}
+
\|v^{*}\|^{2}_{D^{s,2}(\mathbb{R}^{N})}\\
=&
t_{1,v^{*}}^{2_{s}^{*}-2}
\int_{\mathbb{R}^{N}}
|v^{*}|^{2_{s}^{*}}
\mathrm{d}x
+
t_{1,v^{*}}^{p-2}
\beta
\int_{\mathbb{R}^{N}}
|v^{*}|^{p}
\mathrm{d}x
+t_{1,v^{*}}^{2^{*}-2}
\int_{\mathbb{R}^{N}}
|v^{*}|^{2^{*}}
\mathrm{d}x,
\end{aligned}
\end{equation*}
which gives
\begin{equation*}
\begin{aligned}
t_{1,v^{*}}\in(0,1].
\end{aligned}
\end{equation*}
and
\begin{equation*}
\begin{aligned}
\bar{c}_{*}
=
J(t_{1,v^{*}}|v^{*}|)
\leqslant
J(t_{1,v^{*}}v)
\leqslant
\max_{t\geqslant0}
J(tv)
=J(v)
=\bar{c}_{*}.
\end{aligned}
\end{equation*}
Then we know that
$J(t_{1,v^{*}}v^{*})=c$.
From Lemma \ref{Lemma4.5} again, $t_{1,v^{*}}v^{*}$ is a radial ground state solution.
\end{proof}

\section{$L^{\infty}$ estimate}
In this section,
we present the proof of Theorem \ref{Theorem1.3}, i.e., the $L^{\infty}$ estimate of non-negative solution for equations \eqref{F}.
\begin{lemma}\label{Lemma6.1}
Let $N\geqslant 3$, $s\in(0,1)$ and $u\in X$ be non-negative. 
For each $L>2$, we define
\begin{equation*}
\begin{aligned}
u_{L}(x)=
\begin{cases}
u(x)&\mathrm{if}~u(x)\leqslant  L,\\
L&\mathrm{if}~u(x)> L.
\end{cases}
\end{aligned}
\end{equation*}
Then 
$\bar{u}_{L}=uu_{L}^{2(\mu-1)} \in X$ for $\mu>1$.
\end{lemma}

\begin{proof}
For $0\leqslant u\in X$,
by a direct calculation, one has
\begin{equation*}
\begin{aligned}
&\int_{\mathbb{R}^{N}}
\left|\nabla\left(uu_{L}^{2(\mu-1)}\right)\right|^{2}
\mathrm{d}x\\
=&
\int_{\mathbb{R}^{N}}
\left|
u_{L}^{2(\mu-1)}\nabla u
+
u\nabla\left(u_{L}^{2(\mu-1)}\right)\right|^{2}
\mathrm{d}x\\
=&
\int_{\mathbb{R}^{N}}
\left|
u_{L}^{2(\mu-1)}\nabla u
+
2(\mu-1)u u_{L}^{2(\mu-1)-1}\nabla u_{L}\right|^{2}
\mathrm{d}x\\
=&
\int_{\mathbb{R}^{N}}
|u_{L}|^{4(\mu-1)}
|\nabla u|^{2}
\mathrm{d}x
+4(\mu-1)^{2}
\int_{\{u(x)\leqslant L\}}
|u|^{2} |u_{L}|^{4(\mu-1)-2}
|\nabla u_{L}|^{2}
\mathrm{d}x\\
&+4(\mu-1)
\int_{\{u(x)\leqslant L\}}
u u_{L}^{4(\mu-1)-1}\nabla u
\nabla u_{L}
\mathrm{d}x\\
\leqslant&
L^{4(\mu-1)}
\int_{\mathbb{R}^{N}}
|\nabla u|^{2}
\mathrm{d}x
+
L^{4(\mu-1)}
\int_{\{u(x)\leqslant L\}}
|\nabla u_{L}|^{2}
\mathrm{d}x
+
L^{4(\mu-1)}
\int_{\{u(x)\leqslant L\}}
\nabla u
\nabla u_{L}
\mathrm{d}x\\
=&
L^{4(\mu-1)}
\left[
\int_{\mathbb{R}^{N}}
|\nabla u|^{2}
\mathrm{d}x
+2
\int_{\{u(x)\leqslant L\}}
|\nabla u|^{2}
\mathrm{d}x
\right]\\
\leqslant&
3L^{4(\mu-1)}
\int_{\mathbb{R}^{N}}
|\nabla u|^{2}
\mathrm{d}x.
\end{aligned}
\end{equation*}
This shows $\bar{u}_{L}=uu_{L}^{2(\mu-1)} \in D^{1,2}(\mathbb{R}^{N})$.

We now prove $\bar{u}_{L}=uu_{L}^{2(\mu-1)} \in D^{s,2}(\mathbb{R}^{N})$ as follows
\begin{equation*}
\begin{aligned}
&\int_{\mathbb{R}^{N}}
\int_{\mathbb{R}^{N}}
\frac{|uu_{L}^{2(\mu-1)}(x)-uu_{L}^{2(\mu-1)}(y)|^{2}}{|x-y|^{N+2s}}
\mathrm{d}x
\mathrm{d}y
\\
=&
\int_{u(x)\leqslant L}
\int_{u(x)\leqslant L}
\frac{|uu_{L}^{2(\mu-1)}(x)-uu_{L}^{2(\mu-1)}(y)|^{2}}{|x-y|^{N+2s}}
\mathrm{d}x
\mathrm{d}y
\\
&+\int_{u(x)> L}
\int_{u(x)> L}
\frac{|uu_{L}^{2(\mu-1)}(x)-uu_{L}^{2(\mu-1)}(y)|^{2}}{|x-y|^{N+2s}}
\mathrm{d}x
\mathrm{d}y
\\
&+2\int_{u(x)\leqslant L}
\int_{u(x)>L}
\frac{|uu_{L}^{2(\mu-1)}(x)-uu_{L}^{2(\mu-1)}(y)|^{2}}{|x-y|^{N+2s}}
\mathrm{d}x
\mathrm{d}y\\
=:&A_{1}+A_{2}+2A_{3}.
\end{aligned}
\end{equation*}	
In order to compute $A_{1}$, we need
\begin{equation*}
\begin{aligned}
&|u^{2\mu-1}(x)-u^{2\mu-1}(y)|^{2}
-
|u(x)-u(y)|^{2}\\
=&
|u|^{2}(|u|^{4\mu-4}-1)(x)
-
2u(x)u(y)
(|u|^{2\mu-2}(x)|u|^{2\mu-2}(y)-1)
+
|u|^{2}
(|u|^{4\mu-4}-1)(y)\\
\leqslant&
(L^{4\mu-4}-1)
(|u|^{2}(x)-
2u(x)u(y)
+
|u|^{2}(y)),
\end{aligned}
\end{equation*}
which implies
\begin{equation*}
\begin{aligned}
&|u|^{2(2\mu-1)}(x)
-
2|u|^{2\mu-1}(x)|u|^{2\mu-1}(y)
+
u^{2(2\mu-1)}(y)\\
\leqslant&
L^{4\mu-4}|u|^{2}(x)
-
2L^{4\mu-4}u(x)u(y)
+
L^{4\mu-4}
|u|^{2}(y).
\end{aligned}
\end{equation*}
It can be divided into two situations: Case(i)
$u(x)\geqslant u(y)$; Case(ii) 
$u(x)\leqslant u(y)$.
We show Case (i) as follows
\begin{equation*}
\begin{aligned}
&|u|^{2}(x)[L^{4\mu-4}-|u|^{4\mu-4}(x)]
-
2u(x)u(y)
[u^{2\mu-2}(x)u^{2\mu-2}(y)-L^{4\mu-4}]
+
|u|^{2}(y)[L^{4\mu-4}-|u|^{4\mu-4}(y)]
\\
\geqslant&
|u|^{2}(x)[L^{4\mu-4}-|u|^{4\mu-4}(x)]
-
2u(x)u(y)
[u^{4\mu-4}(x)-L^{4\mu-4}]
+
|u|^{2}(y)[L^{4\mu-4}-|u|^{4\mu-4}(x)]
\\
=&
[L^{4\mu-4}-|u|^{4\mu-4}(x)]
(u(x)+u(y))^{2}\geqslant0.
\end{aligned}
\end{equation*}
Similarly, for  Case (ii),
\begin{equation*}
\begin{aligned}
&|u|^{2}(x)[L^{4\mu-4}-|u|^{4\mu-4}(x)]
-
2u(x)u(y)
[u^{2\mu-2}(x)u^{2\mu-2}(y)-L^{4\mu-4}]
+
|u|^{2}(y)[L^{4\mu-4}-|u|^{4\mu-4}(y)]
\\
\geqslant&
[L^{4\mu-4}-|u|^{4\mu-4}(y)]
(u(x)+u(y))^{2}\geqslant 0.
\end{aligned}
\end{equation*}
Combining the above two cases,
we have
\begin{equation}\label{6.1}
\begin{aligned}
A_{1}=&\int_{u(x)\leqslant L}
\int_{u(x)\leqslant L}
\frac{|uu_{L}^{2(\mu-1)}(x)-uu_{L}^{2(\mu-1)}(y)|^{2}}{|x-y|^{N+2s}}
\mathrm{d}x
\mathrm{d}y\\
=&
\int_{u(x)\leqslant L}
\int_{u(x)\leqslant L}
\frac{|u^{2\mu-1}(x)-u^{2\mu-1}(y)|^{2}}{|x-y|^{N+2s}}
\mathrm{d}x
\mathrm{d}y\\
=&
\int_{u(x)\leqslant L}
\int_{u(x)\leqslant L}
\frac{|u|^{2(2\mu-1)}(x)
-
2|u|^{2\mu-1}(x)|u|^{2\mu-1}(y)
+
u^{2(2\mu-1)}(y)}{|x-y|^{N+2s}}
\mathrm{d}x
\mathrm{d}y\\
\leqslant&
\bar{C}\int_{u(x)\leqslant L}
\int_{u(x)\leqslant L}
\frac{|u|^{2}(x)
-
2u(x)u(y)
+
|u|^{2}(y)}{|x-y|^{N+2s}}
\mathrm{d}x
\mathrm{d}y\\
=&
\bar{C}\int_{u(x)\leqslant L}
\int_{u(x)\leqslant L}
\frac{|u(x)-u(y)|^{2}}{|x-y|^{N+2s}}
\mathrm{d}x
\mathrm{d}y.
\end{aligned}
\end{equation}
It is easy to see $A_{2}$ that
\begin{equation}\label{6.2}
\begin{aligned}
A_{2}=&\int_{u(x)> L}
\int_{u(x)> L}
\frac{|uu_{L}^{2(\mu-1)}(x)-uu_{L}^{2(\mu-1)}(y)|^{2}}{|x-y|^{N+2s}}
\mathrm{d}x
\mathrm{d}y\\
=&
\int_{u(x)> L}
\int_{u(x)> L}
\frac{|uL^{2(\mu-1)}(x)-uL^{2(\mu-1)}(y)|^{2}}{|x-y|^{N+2s}}
\mathrm{d}x
\mathrm{d}y\\
=&
L^{4(\mu-1)}
\int_{u(x)> L}
\int_{u(x)> L}
\frac{|u(x)-u(y)|^{2}}{|x-y|^{N+2s}}
\mathrm{d}x
\mathrm{d}y.
\end{aligned}
\end{equation}
Clearly, for
$A_{3}$, we have
\begin{equation}\label{6.3}
\begin{aligned}
&\int_{u(y)\leqslant L}
\int_{u(x)>L}
\frac{|uu_{L}^{2(\mu-1)}(x)-uu_{L}^{2(\mu-1)}(y)|^{2}}{|x-y|^{N+2s}}
\mathrm{d}x
\mathrm{d}y\\
=&
\int_{u(y)\leqslant L}
\int_{u(x)> L}
\frac{|u|^{2}|u_{L}|^{4(\mu-1)}(x)
+
|u|^{2}|u_{L}|^{4(\mu-1)}(y)}{|x-y|^{N+2s}}
\mathrm{d}x
\mathrm{d}y\\
\leqslant&
L^{4(\mu-1)}
\int_{u(y)\leqslant L}
\int_{u(x)> L}
\frac{|u|^{2}(x)
+
|u|^{2}(y)}{|x-y|^{N+2s}}
\mathrm{d}x
\mathrm{d}y\\
=&
L^{4(\mu-1)}
\int_{u(y)\leqslant L}
\int_{u(x)> L}
\frac{|u(x)
-
u(y)|^{2}}{|x-y|^{N+2s}}
\mathrm{d}x
\mathrm{d}y.
\end{aligned}
\end{equation}
Using \eqref{6.1}-\eqref{6.3},
one knows that
\begin{equation*}
\begin{aligned}
&\int_{\mathbb{R}^{N}}
\int_{\mathbb{R}^{N}}
\frac{|uu_{L}^{2(\mu-1)}(x)-uu_{L}^{2(\mu-1)}(y)|^{2}}{|x-y|^{N+2s}}
\mathrm{d}x
\mathrm{d}y
\leqslant
C
\int_{\mathbb{R}^{N}}
\int_{\mathbb{R}^{N}}
\frac{|u(x)
-
u(y)|^{2}}{|x-y|^{N+2s}}
\mathrm{d}x
\mathrm{d}y.
\end{aligned}
\end{equation*}
This shows that $\bar{u}_{L}=uu_{L}^{2(\mu-1)} \in D^{s,2}(\mathbb{R}^{N})$.
\end{proof}

\begin{lemma}\label{Lemma6.2}
Let $N\geqslant 3$, $s\in(0,1)$,
$2_{s}^{*}<p<2^{*}$ and $u\in X$ be a non-negative non-trivial solution of equation \eqref{F}. 
Then we have
\begin{equation*}
\begin{aligned}
\|uu_{L}^{\mu-1}\|^{2}_{L^{t}(\mathbb{R}^{N})}
\leqslant
C\mu^{2}
\left[
\int_{\mathbb{R}^{N}}
|u|^{2_{s}^{*}}|u_{L}|^{2(\mu-1)}\mathrm{d}x
+\int_{\mathbb{R}^{N}}
|u|^{2^{*}}|u_{L}|^{2(\mu-1)}\mathrm{d}x
\right],
\end{aligned}
\end{equation*}
where $t\in [2_{s}^{*},2^{*}]$.
\end{lemma}

\begin{proof}
Let $u$ be a non-negative solution of equation \eqref{F}.
Then
\begin{equation}\label{6.4}
\begin{aligned}
\int_{\mathbb{R}^{N}}
\nabla u
\nabla\phi
\mathrm{d}x
+
\int_{\mathbb{R}^{N}}
\int_{\mathbb{R}^{N}}
\frac{(u(x)-u(y))(\phi(x)-\phi(y))}{|x-y|^{N+2s}}
\mathrm{d}x
\mathrm{d}y
=\int_{\mathbb{R}^{N}}
f(u)\phi
\mathrm{d}x.
\end{aligned}
\end{equation}
From Lemma \ref{Lemma6.1}, we know
\begin{equation}\label{6.5}
\begin{aligned}
\int_{\mathbb{R}^{N}}
\nabla u
\nabla \bar{u}_{L}
\mathrm{d}x
+
\int_{\mathbb{R}^{N}}
\int_{\mathbb{R}^{N}}
\frac{(u(x)-u(y))
(\bar{u}_{L}(x)-\bar{u}_{L}(y))}{|x-y|^{N+2s}}
\mathrm{d}x
\mathrm{d}y
=
\int_{\mathbb{R}^{N}}
f(u)\bar{u}_{L}
\mathrm{d}x.
\end{aligned}
\end{equation}
Let us define
\begin{equation*}
\begin{aligned}
\Lambda
:=
\frac{|t|^{2}}{2}
~\mbox{and}~
Y(t)
:=
\int_{0}^{t}
(\bar{u}'_{L}(\tau))^{\frac{1}{2}}\mathrm{d}\tau.
\end{aligned}
\end{equation*}
It follows from  
$\bar{u}_{L}$ 
is an increasing function that
\begin{equation*}
\begin{aligned}
(a-b)(\bar{u}_{L}(a)-\bar{u}_{L}(b))\geqslant0
~\mbox{for any}~
a,b\in\mathbb{R}^{N}.
\end{aligned}
\end{equation*}
Using Jensen's inequality, we obtain
\begin{equation*}
\begin{aligned}
\Lambda'(a-b)
(\bar{u}_{L}(a)-\bar{u}_{L}(b))
\geqslant
|Y(a)-Y(b)|^{2}
~\mbox{for any}~
a,b\in\mathbb{R}^{N},
\end{aligned}
\end{equation*}
from which it follows that
\begin{equation*}
\begin{aligned}
|Y(u)(x)-Y(u)(y)|^{2}
\leqslant
(u(x)-u(y))
((uu_{L}^{2(\mu-1)})(x)
-(uu_{L}^{2(\mu-1)})(y)).
\end{aligned}
\end{equation*}
Note that 
$Y(u)\geqslant\frac{1}{\mu}uu_{L}^{\mu-1}$.
By Lemma \ref{Lemma2.3}, one has
\begin{equation}\label{6.6}
\begin{aligned}
&\int_{\mathbb{R}^{N}}
\nabla u
\nabla \bar{u}_{L}
\mathrm{d}x
+
\int_{\mathbb{R}^{N}}
\int_{\mathbb{R}^{N}}
\frac{(u(x)-u(y))
(\bar{u}_{L}(x)-\bar{u}_{L}(y))}{|x-y|^{N+2s}}
\mathrm{d}x
\mathrm{d}y\\
\geqslant&
\|Y(u)\|^{2}_{D^{1,2}(\mathbb{R}^{N})}
+
\|Y(u)\|^{2}_{D^{s,2}(\mathbb{R}^{N})}\\
\geqslant&
C
\|Y(u)\|^{2}_{L^{t}(\mathbb{R}^{N})}\\
\geqslant&
\frac{C}{\mu^{2}}\|uu_{L}^{\mu-1}\|^{2}_{L^{t}(\mathbb{R}^{N})}.
\end{aligned}
\end{equation}
By virtue of \eqref{6.5}, \eqref{6.6} and $(F_{1})$, we have
\begin{equation*}
\begin{aligned}
\|uu_{L}^{\mu-1}\|^{2}_{L^{t}(\mathbb{R}^{N})}
\leqslant&
C\mu^{2}
\left(\int_{\mathbb{R}^{N}}
\nabla u
\nabla \bar{u}_{L}
\mathrm{d}x
+
\int_{\mathbb{R}^{N}}
\int_{\mathbb{R}^{N}}
\frac{(u(x)-u(y))
(\bar{u}_{L}(x)-\bar{u}_{L}(y))}{|x-y|^{N+2s}}
\mathrm{d}x
\mathrm{d}y\right)\\
=&
C\mu^{2}
\int_{\mathbb{R}^{N}}
f(u)\bar{u}_{L}
\mathrm{d}x\\
\leqslant&C\mu^{2}
\left[
\int_{\mathbb{R}^{N}}
|u|^{2_{s}^{*}}|u_{L}|^{2(\mu-1)}\mathrm{d}x
+\int_{\mathbb{R}^{N}}
|u|^{2^{*}}|u_{L}|^{2(\mu-1)}\mathrm{d}x
\right].
\end{aligned}
\end{equation*}
The proof is completed.
\end{proof}

We are now ready to show the $L^{\infty}$ estimation of solution $u$.

\begin{proof}[The proof of Theorem \ref{Theorem1.3}] 
{\bf Step 1.}
From Lemma \ref{Lemma6.2}, we have
\begin{equation*}
\begin{aligned}
\|uu_{L}^{\mu-1}\|^{2}_{L^{2^{*}}(\mathbb{R}^{N})}
\leqslant
C\mu^{2}
\left[
\int_{\mathbb{R}^{N}}
|u|^{2_{s}^{*}}|u_{L}|^{2(\mu-1)}\mathrm{d}x
+
\int_{\mathbb{R}^{N}}
|u|^{2^{*}}|u_{L}|^{2(\mu-1)}\mathrm{d}x
\right].
\end{aligned}
\end{equation*}
Let $\bar{d}\in \mathbb{R}^{+}$ be chosen later.
By using H\"{o}lder's inequality,
we have
\begin{equation*}
\begin{aligned}
&\int_{\mathbb{R}^{N}}
|u|^{2^{*}-2}|uu_{L}^{\mu-1}|^{2}
\mathrm{d}x\\
\leqslant&
\int_{\{u(x)\leqslant \bar{d}\}}
|u|^{2^{*}-2}|uu_{L}^{\mu-1}|^{2}
\mathrm{d}x
+
\int_{\{u(x)>\bar{d}\}}
|u|^{2^{*}-2}|uu_{L}^{\mu-1}|^{2}
\mathrm{d}x\\
\leqslant&
\bar{d}^{2^{*}-2_{s}^{*}}
\int_{\{u(x)\leqslant \bar{d}\}}
|u|^{2_{s}^{*}-2}
|uu_{L}^{\mu-1}|^{2}
\mathrm{d}x
+
\int_{\{u(x)>\bar{d}\}}
|u|^{2^{*}-2}|uu_{L}^{\mu-1}|^{2}
\mathrm{d}x\\
\leqslant&
\bar{d}^{2^{*}-2_{s}^{*}}
\int_{\{u(x)\leqslant \bar{d}\}}
|u|^{2_{s}^{*}-2}
|uu_{L}^{\mu-1}|^{2}
\mathrm{d}x
+
\left(
\int_{\{u(x)>\bar{d}\}}
|u|^{2^{*}}
\mathrm{d}x
\right)^{\frac{2^{*}-2}{2^{*}}}
\left(
\int_{\{u(x)>\bar{d}\}}
|uu_{L}^{\mu-1}|^{2^{*}}
\mathrm{d}x
\right)^{\frac{2}{2^{*}}}.
\end{aligned}
\end{equation*}
We can fix $\bar{d}$ such that
\begin{equation*}
\begin{aligned}
\left(
\int_{\{u(x)>\bar{d}\}}
|u|^{2^{*}}
\mathrm{d}x
\right)^{\frac{2^{*}-2}{2^{*}}}
\leqslant
\frac{1}{2C\mu^{2}}.
\end{aligned}
\end{equation*}
Then we have
\begin{equation*}
\begin{aligned}
\|uu_{L}^{\mu-1}\|^{2}_{L^{2^{*}}(\mathbb{R}^{N})}
\leqslant
C\mu^{2}(1+\bar{d}^{2^{*}-2_{s}^{*}})
\int_{\mathbb{R}^{N}}
|u|^{2_{s}^{*}}|u_{L}|^{2(\mu-1)}\mathrm{d}x.
\end{aligned}
\end{equation*}
Taking the limit as
$L\rightarrow\infty$
in the above inequality,
we deduce
\begin{equation}\label{6.7}
\begin{aligned}
\left(
\int_{\mathbb{R}^{N}}
|u|^{2^{*}\mu}
\mathrm{d}x
\right)^{\frac{2}{2^{*}}}
\leqslant
C\mu^{2}(1+\bar{d}^{2^{*}-2_{s}^{*}})
\int_{\mathbb{R}^{N}}
|u|^{2_{s}^{*}+2\mu-2}\mathrm{d}x.
\end{aligned}
\end{equation}

{\bf Step 2.}
Let $\mu_{1}-1=\frac{2^{*}-2_{s}^{*}}{2}$. Then from \eqref{6.7}, we have
\begin{equation*}
\begin{aligned}
\left(
\int_{\mathbb{R}^{N}}
|u|^{2^{*}\mu_{1}}
\mathrm{d}x
\right)^{\frac{2}{2^{*}}}
\leqslant
C\mu_{1}^{2}(1+\bar{d}^{2^{*}-2_{s}^{*}})
\int_{\mathbb{R}^{N}}
|u|^{2^{*}}\mathrm{d}x<+\infty,
\end{aligned}
\end{equation*}
which gives 
\begin{equation*}
\begin{aligned}
\left(
1+
\int_{\mathbb{R}^{N}}
|u|^{2^{*}\mu_{1}}
\mathrm{d}x
\right)^{\frac{2}{2^{*}(\mu_{1}-1)}}<+\infty.
\end{aligned}
\end{equation*}

{\bf Step 3.} We choose $\mu_{2}-1=\frac{2^{*}}{2}(\mu_{1}-1)$. Then 
\begin{equation*}
\begin{aligned}
\left(
\int_{\mathbb{R}^{N}}
|u|^{2^{*}\mu_{2}}
\mathrm{d}x
\right)^{\frac{2}{2^{*}}}
\leqslant
C\mu_{2}^{2}
\left[
\int_{\mathbb{R}^{N}}
|u|^{2_{s}^{*}+2\mu_{2}-2}\mathrm{d}x
+
\int_{\mathbb{R}^{N}}
|u|^{2^{*}+2\mu_{2}-2}\mathrm{d}x
\right].
\end{aligned}
\end{equation*}
Let 
\begin{equation*}
\begin{aligned}
a=\frac{2^{*}(2^{*}-2_{s}^{*})}{2(\mu_{2}-1)},
~~
b=2_{s}^{*}+2\mu_{2}-2-a.
\end{aligned}
\end{equation*}
Then $\frac{2^{*}b}{2^{*}-a}=2^{*}+2\mu_{2}-2$. By applying Young's inequality, we have
\begin{equation*}
\begin{aligned}
\int_{\mathbb{R}^{N}}
|u|^{2_{s}^{*}+2\mu_{2}-2}\mathrm{d}x
\leqslant&
\frac{a}{2^{*}}
\int_{\mathbb{R}^{N}}
|u|^{2^{*}}\mathrm{d}x
+
\frac{2^{*}-a}{2^{*}}
\int_{\mathbb{R}^{N}}
|u|^{2^{*}+2\mu_{2}-2}\mathrm{d}x\\
\leqslant&
C\left(
1+
\int_{\mathbb{R}^{N}}
|u|^{2^{*}+2\mu_{2}-2}\mathrm{d}x\right).
\end{aligned}
\end{equation*}
Thus
\begin{equation*}
\begin{aligned}
\left(
\int_{\mathbb{R}^{N}}
|u|^{2^{*}\mu_{2}}
\mathrm{d}x
\right)^{\frac{2}{2^{*}}}
\leqslant
C\mu_{2}^{2}
\left(
1+
\int_{\mathbb{R}^{N}}
|u|^{2^{*}+2\mu_{2}-2}\mathrm{d}x\right).
\end{aligned}
\end{equation*}
Using $\mu_{2}>1$, we deduce
\begin{equation*}
\begin{aligned}
\left(1+
\int_{\mathbb{R}^{N}}
|u|^{2^{*}\mu_{2}}
\mathrm{d}x
\right)^{\frac{2}{2^{*}}}
\leqslant&
1+\left(
\int_{\mathbb{R}^{N}}
|u|^{2^{*}\mu_{2}}
\mathrm{d}x
\right)^{\frac{2}{2^{*}}}\\
\leqslant&
C\mu_{2}^{2}
\left(
1+
\int_{\mathbb{R}^{N}}
|u|^{2^{*}+2\mu_{2}-2}\mathrm{d}x\right).
\end{aligned}
\end{equation*}
Then 
\begin{equation*}
\begin{aligned}
\left(1+
\int_{\mathbb{R}^{N}}
|u|^{2^{*}\mu_{2}}
\mathrm{d}x
\right)^{\frac{2}{2^{*}(\mu_{2}-1)}}
\leqslant&
(C\mu_{2}^{2})^{\frac{1}{\mu_{2}-1}}
\left(
1+
\int_{\mathbb{R}^{N}}
|u|^{2^{*}+2\mu_{2}-2}\mathrm{d}x\right)^{\frac{1}{\mu_{2}-1}}\\
\leqslant&(C\mu_{2})^{\frac{2}{\mu_{2}-1}}
\left(
1+
\int_{\mathbb{R}^{N}}
|u|^{2^{*}\mu_{1}}\mathrm{d}x\right)^{\frac{2}{2^{*}(\mu_{1}-1)}}
\end{aligned}
\end{equation*}

{\bf Step 4.}
Iterating this process and recalling that
\begin{equation*}
\begin{aligned}
\mu_{i+1}-1=\frac{2^{*}}{2}(\mu_{i}-1),
\end{aligned}
\end{equation*}
and
\begin{equation*}
\begin{aligned}
\left(1+
\int_{\mathbb{R}^{N}}
|u|^{2^{*}\mu_{i+1}}
\mathrm{d}x
\right)^{\frac{2}{2^{*}(\mu_{i+1}-1)}}
\leqslant&
(C\mu_{i+1})^{\frac{2}{\mu_{i+1}-1}}
\left(
1+
\int_{\mathbb{R}^{N}}
|u|^{2^{*}\mu_{i}}\mathrm{d}x\right)^{\frac{2}{2^{*}(\mu_{i}-1)}}.
\end{aligned}
\end{equation*}
Denoting $C_{i+1}:=C\mu_{i+1}$ and $K_{i}=\left(
1+
\int_{\mathbb{R}^{N}}
|u|^{2^{*}\mu_{i}}\mathrm{d}x\right)^{\frac{2}{2^{*}(\mu_{i}-1)}}$, we know that there exists a constant $C > 0$ independent of $i$ such that
\begin{equation*}
\begin{aligned}
K_{m+1}\leqslant \prod_{i=1}^{m}
C_{i+1}^{\frac{2}{\mu_{i+1}-1}}
K_{1}\leqslant CK_{1}<+\infty.
\end{aligned}
\end{equation*}
Therefore,
\begin{equation*}
\begin{aligned}
\|u\|_{L^{\infty}(\mathbb{R}^{N})}\leqslant CK_{1}<+\infty.
\end{aligned}
\end{equation*}
\end{proof}

\section*{Conflict of interest}
The author declares that he has no known competing financial interests or personal relationships that could have appeared to influence the work reported in
this article.

\section*{Data availability statement}
No data were used for the research described in the article.

\end{document}